\documentclass{amsart}

\usepackage{amsmath,amssymb,url}

\numberwithin{equation}{section}

\newcommand{\pol}{\text{Pol}}

\theoremstyle{plain}
\newtheorem{theorem}{Theorem}[section]
\newtheorem{lemma}[theorem]{Lemma}
\newtheorem{proposition}[theorem]{Proposition}
\newtheorem{example}[theorem]{Example}

\theoremstyle{definition}
\newtheorem{definition}[theorem]{Definition}

\newtheorem{remark}[theorem]{Remark}

\begin{document}


\title[Meet-reducible submaximal clones and nontrivial equivalence relations]{Meet-reducible submaximal clones determined by nontrivial equivalence relations.}

\author[L. E. F. Di\'ekouam]{Luc E. F. Di\'ekouam}
\email{lucdiekouam@yahoo.fr} 

\address{Department of mathematics\\Ecole Normale Sup\'erieure\\
University of Maroua \\\text{   }\hspace{2cm} P.O. Box 55 Maroua
\\Cameroon}

\author[E. R. A. Temgoua]{\'Etienne R. A. Temgoua}
\email{retemgoua@yahoo.fr}

\address{Department of Mathematics\\Ecole Normale Sup\'erieure\\University of Yaound\'e-1\\\text{   }\hspace{2cm}
P.O. Box 47 Yaound\'e\\Cameroon}

\author[M. TONGA]{Marcel TONGA}
\email{tongamarcel@yahoo.fr}

\address{Department of Mathematics\\Faculty of Sciences\\University of Yaound\'e-1\\
P.O. Box 812 Yaound\'e\\Cameroon}




\begin{abstract}
The structure of the lattice of clones on a finite set has been
proven to be very complex. To better understand the top of this
lattice, it is important to provide a characterization of submaximal
clones in the  lattice of clones. It is known that the clones
$\pol(\theta)$ and $\pol(\rho)$ (where $\theta$ is a nontrivial
equivalence relation on $E_k = \{0, . . . ,k - 1\}$, and $\rho$ is
among the six types of relations which characterize maximal clones)
are maximal clones. In this paper, we provide a classification of
relations (of Rosenberg's List) $\rho$ on $E_k$ such that the clone
$\pol(\theta) \cap \pol(\rho)$ is maximal in $\pol(\theta)$.
\end{abstract}
\maketitle
\section{Introduction}
The structure of the lattice of clones on a finite set of more than
two elements is quite complex. An indication of such a complexity is
through its cardinality, which is $2^{\aleph_0}$. For a better
picture of some intervals in this lattice, it is important to
provide a characterization of maximal and submaximal clones. Maximal
clones have been investigated extensively by I. G. Rosenberg, and a
complete characterization of these can be found in \cite{ROS65}.
More precisely, it is proved that for a given nontrivial equivalence
relation $\theta$ on a finite set and a central relation $\rho$ on
the same set, the clones $\pol(\theta)$ and $\pol(\rho)$ are
maximal. For a unary central relation on an arbitrary finite set,
Rosenberg and Szendrei \cite{ROS70-2, ROS85} investigated the
submaximal clones of their polymorphisms and obtain new results on
polymorphism of prime permutations on a finite set. Submaximal
clones for a set with two and three elements were completely
described and classified in \cite{LAU,PO21,PO41}. However, for sets
with more than three elements, only partial results on their
submaximal clones are found in the literature (see for e.g.,
\cite{LAU}). Recently, Temgoua and Rosenberg \cite{TEM-ROS} obtained
a characterization of all binary central relations such that the
clone $\pol(\theta)\cap\pol(\rho)$ is maximal in $\pol(\theta)$,
given any nontrivial equivalence relation $\theta$ and a binary
central relation $\rho$.

In this paper, we characterize all  relations $\rho$ such that the clones of the form $\pol(\theta) \cap \pol(\rho)$ is maximal in $\pol(\theta)$, where
$\theta$ is a nontrivial equivalence relation on a given finite set.

The rest of the paper is organized as follows: Section \ref{sec2}
recalls the necessary basic definitions and notations for the
clarity of our presentation. In Section \ref{sec3}, we present the
no submaximality when $\rho$ is a partial order or a prime affine
relation. Section \ref{sec4} is devoted to the characterization of
type of equivalence relations or prime permutation relations which
give submaximality. In the  Section \ref{sec5}, we characterize the
relations $\rho$ (resp. central relations or $h$-regularly generated
relations ) for which $\pol(\theta) \cap \pol(\rho)$ is maximal in
$\pol(\theta)$. Section \ref{sec8} concludes the paper.

\section{Preliminaries}\label{sec2}
In this section, we recall some of the main definitions and notations. Readers needing more background on the topic are encouraged to consult \cite{LAU}.
Let $E_k=\{0,1,\hdots,k-1\}$ be a finite set of
$k$ elements with $k\geq 3$. Let $n,s\in\mathbb{N}^{*}$. An $n$-ary operation on $E_k$ is a function from  $E_k^n$ to $E_k.$ The set of all $n$-ary
operations on $E_k$ is denoted by $\mathcal{O}^n(E_k)$ and we set $\mathcal{O}(E_k)=\bigcup\limits_{0<n<\omega}\mathcal{O}^n(E_k).$ For $1\leq i \leq s$,
the $s$-ary $i$-th projection $e_i^s$ is defined as $e_i^s(x_1,\hdots,x_s)=x_i$ for all $x_1,\hdots,x_s.$ For $f\in \mathcal{O}^n(E_k)$ and $g_1,\hdots,
g_n\in \mathcal{O}^m(E_k)$, we define their composition to be the $m$-ary operation $f[g_1,\hdots,g_n]$ defined by:
$$f[g_1,\hdots,g_n](x_1,\hdots,x_m)=f(g_1(x_1,\hdots,x_m),\hdots,g_n(x_1,\hdots,x_m)).$$
A clone on $E_k$ is a subset $F$ of $\mathcal{O}(E_k)$ which contains all the projections and is closed under composition. It is known that the
intersection of an arbitrary set of clones on $E_k$ is a clone on $E_k.$ Thus for $F\subseteq \mathcal{O}(E_k)$, there exists a smallest clone containing
$F$, called the clone generated by $F$ and denoted by $\langle F\rangle.$  $\langle F\rangle$ is also the set of term operations of the non-indexed algebra
$\mathcal{A}=(A;F)$, with $A=E_{k}$. The clones on $E_k$, ordered by inclusion, form a complete lattice denoted by $\mathcal{L}(E_k).$ A clone
$C\in\mathcal{L}(E_k)$ is called maximal if it is covered only by $\mathcal{O}(E_k)$. A clone $C\in\mathcal{L}(E_k)$ is called \textit{submaximal} if it is
covered only by a maximal clone.

Let $h$ be a positive integer. An $h$-ary relation $\rho$ is a subset of $E_k^h.$ For $\rho\subseteq E_k^2$, we write $a \,\rho\, b$ for $(a,b)\in \rho.$
An $h$-ary relation $\rho$ is called totally symmetric if for every permutation $\sigma$ of $\{1,\hdots,h\}$ and each $h$-tuple $(a_1,\hdots, a_h)\in
E_k^h$,
$$(a_1,\hdots, a_h)\in\rho \text{ if } (a_{\sigma(1)},\hdots, a_{\sigma(h)})\in\rho.$$ $\tau_h^{E_k}$ is the $h$-ary relation defined by $(a_1,\hdots,
a_h)\in\tau_h^{E_k}$ if there exist $i,j\in \{1,\hdots,h\}$ such that $i\neq j$ and $a_i=a_j.$ An $h$-ary relation $\rho$ is called totally reflexive if
$\tau_h^{E_k}\subseteq \rho.$ For $h=2$, the concepts totally reflexive and totally symmetric coincide with the usual notions of reflexive and symmetric.
If $\rho$ is totally reflexive and totally symmetric, we define the center of $\rho$ denoted by $C_{\rho}$ as follows: $$C_{\rho}=\{a\in E_k:
(a,a_2,\hdots,a_h)\in \rho \text{ for all } a_2,\hdots,a_h\in E_k\}.$$

Let $\theta$ be a binary relation and $m\in\mathbb{N}^{*}$; for
$\textbf{a} = (a_1, \ldots, a_m)\in E^{m}_{k}$ and $\textbf{b} =
(b_1,\ldots, b_m)\in E_{k}^{m}$, we write
$\textbf{a}\theta\textbf{b}$ if $(a_i,b_i) \in \theta$ for $1 \leq i
\leq m$. Assume that $\theta$ is an equivalence relation on $E_{k}$.
The $\theta$-class of $a\in E_k$ will be denoted by $[a]_{\theta}$.

A permutation $\pi$ on $E_{k}$ is prime if all cycles of $\pi$ have the same prime length.

A subset $\sigma\subseteq E_{k}^{4}$ is called affine if there is a binary operation $+$ on $E_{k}$ such that $(E_{k},+)$ is an abelian group and $(a, b,
c, d)\in\sigma\Leftrightarrow a + b = c + d$. An affine relation $\sigma$ is prime if $(E_{k},+)$ is an abelian $p$-group for some prime $p$, that is, all
elements of the group have the same prime order $p$.

An $h$-ary relation $\rho$ on $E_k$ is called central, if $\rho$ is
a nonempty proper subset of $E_{k}$ or $\rho$ has the  following
three properties: $\rho$ is totally reflexive; $\rho$ is totally
symmetric and   $ C_{\rho}$ is a nonempty proper subset of $E_{k}$.

For $h\geq3$, a family $T =
\{\mathcal{V}_{1};\cdots\mathcal{V}_{m}\}$ of equivalence relations
on $E_{k}$ is called $h$-regular if each $\mathcal{V}_{i}$ has
exactly $h$ equivalence classes and $\cap\{B_i| 1\leq i\leq m\}$ is
nonempty for arbitrary equivalence classes $B_i$ of
$\mathcal{V}_{i}$ , $1\leq i\leq m$. For $3 \leq h\leq k$, an
$h$-regular (or $h$-regularly generated) relation on $E_{k}$
determined by the $h$-regular family $T$(often denoted by
$\lambda_{T}$), consists of all $h$-tuples whose set of components
meets at most $h-1$ classes of each $\mathcal{V}_{i}$ ( $1\leq i\leq
m)$.

Let $f\in\mathcal{O}^n(E_k)$ and  $\rho$ be an $h$-ary relation on
$E_{k}$. The operation $f$ preserves $\rho$ if for all $(a_{1,i},
\hdots , a_{h,i})\in\rho$ $(i = 1,\ldots, n)$, we have
$$(f(a_{1,1},\hdots , a_{1,n}), f(a_{2,1}, \hdots, a_{2,n}), \hdots
, f(a_{h,1}, \hdots , a_{h,n}))\in\rho.$$ The set of operations on
$E_{k}$ preserving $\rho$ is a clone denoted by $\text{Pol}(\rho)$.
The maximal clones have been described for $k = 2$(respectively $k =
3$ and $k \geq 4$) in Post\cite{PO21}(respectively Jablonskij
\cite{JAB} and  Rosenberg \cite{ROS65,ROS70}). They are of the form
$\text{Pol}(\rho)$ where $\rho$ belongs to one of six families of
relations which include some familiar and easily defined relations.
For clones $C$ and $D$ on $E_{k}$, we say that $C$ is maximal in $D$
if $D$ covers $C$ in $\mathcal{L}(E_k)$; we also say that $C$ is
submaximal if $C$ is maximal in at least one maximal clone. All
submaximal clones are known for $k = 2$ (see \cite{PO21}) and $k =
3$ (see \cite{LAU}).

For $n \geq 3$, an $n$-ary operation $f$ is called a near-unanimity operation provided that $f(y,x,\hdots,x)\approx f(x,y,\hdots,x)\approx\hdots \approx
f(x,x,\hdots,y)\approx x$ for all $x,y\in E_{k}.$ We recall the following Baker-Pixley Theorem which will be used to prove our results:

\begin{theorem}\label{baker-pixley} {\cite{BA-PI}}
Let $\mathcal{A}=(A,F)$ be a finite algebra which contains a ``near unanimity function" of arity $d+1$ ($(d+1)$-ary near-unanimity term or $nu$-term).
Then, an operation $f:A^{n}\rightarrow A$ is term function for  $\mathcal{A}$ iff  each subuniverse of $\mathcal{A}^{d}$ is preserved by $f$.
\end{theorem}

\section{Partial order, prime affine relations}\label{sec3}
In this section we prove that the clones $\pol(\theta) \cap
\pol(\rho)$ is not maximal in $\pol(\theta)$ where $\theta$ is a
nontrivial equivalence relation on $E_{k}$ and $\rho$ is either a
partial order with least and greatest elements or a prime affine
relation on $E_{k}$.
\begin{theorem}\label{maintheo-order-equiv}
If $ \theta $  is a nontrivial equivalence relation and $ \rho $ is a partial order with least and greatest elements, on a finite set $E_{k} $, then $
\pol(\theta) \cap \pol(\rho)$ is not submaximal in $\pol(\theta)$.
\end{theorem}
\begin{proof}
$\theta$ and $\rho$ are incomparable. In fact,
$\theta\nsubseteq\rho$ because $\theta $ is a nontrivial symmetric
relation and $\rho$ is an antisymmetric relation. We also have
$\rho\nsubseteq\theta$, otherwise since $\rho$ is a partial order
with least and greatest element and $\theta$ is a transitive
relation, $\theta$ will be trivial (equal to $E_{k}^{2}$); this is a
contradiction. Without lost of generality we can consider that the
least element of $\rho$ is $0$ and the greatest element of $\rho$ is
$1$.

Let $\rho'$ and $r$ be the relations defined by: $\rho':=\theta\circ\rho\circ\theta$ and $r:=\rho\cap\theta$.

If $\rho'\neq E_{k}^{2}$, we have $\pol(\theta)\cap\pol(\rho)\varsubsetneq\pol(\rho')\varsubsetneq\pol(\theta)$.

Suppose that $\rho'= E_{k}^{2}$, then $(1,0)\in E_{k}^{2}=\rho'$ and
it follows that $(1,0)\in\theta$ or there is $b\in
E_{k}\setminus\{0\}$ such that $(b,0)\in\theta$. Therefore $r\neq
\Delta_{E_{k}}=\{(a,a):\ \ a\in E_{k}\}$.

Since $r$ and $r^{-1}$ are subset of $\theta$,  $r\circ r^{-1}$ is a subset of $\theta$.

If $r\circ r^{-1}\neq \theta$, then
$\pol(\theta)\cap\pol(\rho)\varsubsetneq\pol(\theta)\cap\pol(r\circ
r^{-1})\varsubsetneq\pol(\theta)$.

If $r\circ r^{-1}= \theta$, then $\pol(r)\varsubsetneq\pol(\theta)$
and it can be proved that each equivalence class of $\theta$
contains a least and a greatest element. It follows by Theorem 3.3
of \cite{TEM-meet-irred}  that $\pol(r)$ is a meet-irreducible
maximal subclone of $\pol(\theta)$. Using the fact that
$\pol(\theta)\cap\pol(\rho)\varsubsetneq \pol(r)$, we conclude that
$\pol(\theta)\cap\pol(\rho)$ is not maximal in $\pol(\theta)$.
\end{proof}

For a prime affine relation, there is no meet-reducible
submaximality in the set of polymorphisms of a nontrivial
equivalence relation. This is proved by the following theorem.
\begin{theorem}\label{maintheo-primeaffine-equiv}
Let $\alpha$ be a prime affine relation and $\theta$ a nontrivial equivalence relation. We have
$$ \pol(\theta)\cap\pol(\alpha)\varsubsetneq\pol(\alpha_{1})\varsubsetneq\pol(\theta), $$ where
$$\alpha_1 = \{(a, b, c, d)\in\alpha:  (a, b),(a,c),(a,d) \in \theta\}.$$
\end{theorem}
\begin{proof}

 Let $(a,b)\in\theta$ such that $a\neq b$, then the binary operation $g_1$ defined on $E_k$ by:
 \begin{displaymath}
  g_1(x,y)=\left\{ \begin{array}{l l}
   a & \textrm{if}\;(x,y)\in\{(a,a);(a,b);(b,a)\}\\
   b & \textrm{otherwise,}
  \end{array} \right.
 \end{displaymath}
 preserves $\theta$ and does not preserve $\alpha_{1}$. Therefore $\pol(\alpha_{1})\varsubsetneq \pol(\theta)$.

 Also it is easy to see that $\pol(\theta)\cap\pol(\alpha)\subseteq\pol(\alpha_{1})$. To continue, let $(a,b)\notin\theta$ and let $g_{2}$ be a ternary operation  on
 $E_{k}$ defined by:
 \begin{displaymath}
  g_2(x,y,z)=\left\{ \begin{array}{l l}
   b & \textrm{if}\;(x,y,z)\theta(a,a,b)\\
   a & \textrm{otherwise.}
  \end{array} \right.
 \end{displaymath}
 $g_2$ preserves $\alpha_{1}$ and does not preserve $\alpha$; therefore $\pol(\theta)\cap\pol(\alpha)\varsubsetneq \pol(\alpha_{1})$.
\end{proof}

\section{Equivalence relations, prime permutation relations}\label{sec4}
Let $k > 1$, $\theta$ a nontrivial equivalence relation on $E_{k}$
with blocks (equivalence classes) $B_0,\ldots,B_{t_{1}-1}$ (where
$2\leq t_1\leq k)$ and $\rho$ a nontrivial equivalence relation
distinct from $\theta$ with blocks $C_0,\ldots,C_{t_{2}-1}$ (where
$2\leq t_2\leq k$). In this section we determine the meet-reducible
clones of the form $\pol(\theta) \cap \pol(\rho)$ maximal in
$\pol(\theta)$ where $\theta$ and $\rho$ are two distinct nontrivial
equivalence relations on $E_{k}$.

We define $\mu : E_{k} \rightarrow E_{t_{1}}$ by $\mu(x) = i$ if $x
\in B_i$ and $\nu: E_{k} \rightarrow E_{t_{2}}$ by $\nu(x) = i$ if
$x \in C_i$. Set $D = \pol(\theta) \cap \pol(\rho)$, $\nabla_{E_{k}}
=E_{k}^{2}$ and $\gamma = \theta \cap \rho$. Clearly $ \gamma $ is
an equivalence relation on $E_{k}$ and $D\subseteq \pol(\theta) \cap
\pol(\gamma).$
\begin{theorem}\label{maintheo-equiv-equiv} Let $ \theta $ and $ \rho $ be two distinct nontrivial equivalence relations on a finite set $E_{k} $.
$ \pol(\theta) \cap  \pol(\rho)$ is submaximal in  $\pol(\theta)$ if and only if $\theta$ and $\rho$ satisfy one of the following statement:
\begin{enumerate}
    \item [(a)] $\theta \varsubsetneq \rho$ or $\rho \varsubsetneq \theta$;
    \item [(b)] $ \rho\cap\theta=\Delta_{E_{k}} $ and $ \rho\circ\theta=\nabla_{E_{k}}. $
\end{enumerate}
\end{theorem}
The proof of Theorem \ref{maintheo-equiv-equiv} follows from the
next lemmas giving the sufficiency part and the necessity part of
this theorem.
\begin{remark}
In the condition $(b)$, it follows from $
\rho\circ\theta=\nabla_{E_{k}} $ that  there exist $u_i \in B_i$ $(i
= 0,\ldots, t_1-1)$ such that $(u_p,u_q)\in\rho$ for all $0\leq p,$
$q\leq t_{1}-1$. Also  $ \rho\circ\theta=\nabla_{E_{k}} $ is
equivalent to $ \theta\circ\rho=\nabla_{E_{k}}. $
\end{remark}

\begin{lemma}\label{equivequivbin}
Let $\beta$ be a binary relation on $E_k$, $\rho$ and $\theta$ satisfying the condition $(a)$ or $(b)$ of Theorem \ref{maintheo-equiv-equiv}.
$$\text{ If }\pol(\theta) \cap  \pol(\rho)\subseteq \pol(\beta), \text{ then }  \beta\in\{\Delta_{E_{k}};\rho;\theta;\nabla_{E_{k}}\}.$$
\end{lemma}
\begin{proof}\text{ }

\begin{itemize}
    \item Assume that $\rho$ and $\theta$ satisfy condition $(a)$.
    Moreover, suppose that $\rho \varsubsetneq \theta$, $\pol(\theta) \cap  \pol(\rho)\subseteq \pol(\beta)$ and $ \beta\neq\Delta_{E_{k}} $.
     For any distinct $a_1, a_2 \in E_k$ and for any $b_i \in [a_i]_{\rho}$ ($i = 1, 2$), the function $f: E_k \rightarrow E_k$
defined by $f(a_i) = b_i, (i = 1 , 2)$ and $f(x) = x$ for all $x \in
E_k\setminus \{a_1; a_2\}$ satisfies $f \in \pol(\theta) \cap
\pol(\rho)$. Since $\pol(\theta) \cap  \pol(\rho)\subseteq
\pol(\beta)$,  $f \in \pol(\beta).$ It follows that
$[a_1]_{\rho}\times[a_2]_{\rho}\subseteq \beta $ whenever
$(a_{1},a_{2})\in\beta$. In the following we denote by $\beta/\rho $
the relation on $E_{k}/\rho:=\{[a]_{\rho}| a\in E_{k}\}$ defined by:
$([a_1]_{\rho},[a_2]_{\rho})\in\beta/\rho\Leftrightarrow
(a_1,a_2)\in\beta.$ Similarly we define $\rho/\rho $ and
$\theta/\rho $ and we have $\pol(\theta/\rho) \cap
\pol(\rho/\rho)\subseteq \pol(\beta/\rho)$. Since
$\rho/\rho=\Delta_{E_{k}/\rho} $, $\pol(\theta/\rho) \cap
\pol(\rho/\rho)=\pol(\theta/\rho)$ is the maximal clone on
$E_{k}/\rho$ determined by the nontrivial equivalence relation
$\theta/\rho$ and it follows that
$\beta/\rho\in\{\rho/\rho=\Delta_{E_{k}/\rho};\theta/\rho;\nabla_{E_{k}/\rho}=\nabla_{E_{k}}/\rho\}.$
Hence $\beta\in\{\rho;\theta;\nabla_{E_{k}}\}$.
   \item Assume that $\rho$ and $\theta$ satisfy condition $(b)$.
In this case the function $\varphi: E_{k}\rightarrow
E_{k}/\rho\times E_{k}/\theta $, defined by $a\mapsto
([a]_{\rho},[a]_{\theta}) $ is a bijection. This bijection gives a
decomposition of $E_{k}$ into a cartesian product of $E_{k}/\rho$
and $E_{k}/\theta $ and one deduces that the operations in $
\pol(\theta) \cap  \pol(\rho)$ correspond to the operations that act
coordinatewise on $E_{k}/\rho\times E_{k}/\theta.$
\end{itemize}
\end{proof}

\begin{lemma}
Let $\beta$ be a binary relation on $E_k$. If $\rho$ and $\theta$
satisfy condition $(a)$ or $(b)$ in Theorem
\ref{maintheo-equiv-equiv}, then $\pol(\theta) \cap  \pol(\rho)$
contains a majority operation.
\end{lemma}
\begin{proof}\text{ }

\begin{itemize}
    \item Assume that $\rho$ and $\theta$ satisfy condition $(a)$. In addition, assume that $\rho \varsubsetneq \theta$. In each set $B_{i}$, we fix an
    element $v_{B_{i}}$ $(i=0,\ldots,t_{1}-1).$
    We consider the majority operation $m$ defined on $E_{k}$ by:

$$ m(x_1,x_2,x_3)=  \left\{\begin{array}{ll}
                                                 x_{i}& \text{ if }  (x_i,x_j)\in\rho \text{ and } x_l\notin[x_i]_{\rho},  \\
                                                                 & \qquad \text{ for some } 1\leq i<j\leq 3, \text{ such that } l\notin\{i;j\},\\
                                                  x_1& \text{ if } \{x_1;x_2\}\subseteq [x_3]_{\rho},\\
                                                v_{[x_{l}]_{\theta}}& \text{ if }  x_i\notin [x_j]_{\rho} \text{ for } 1\leq i<j\leq 3 \text{ and }\\
                                                                    &\qquad              (x_l,x_m)\in\theta \text{ for some } 1\leq l<m\leq 3,\\
                                                 0 & \text{ otherwise.}
                                               \end{array}
      \right.$$
Let us show that $m\in \pol(\rho)\cap\pol(\theta) $.

Let $(a_i,b_i)\in\rho, 1\leq i\leq 3 $

\begin{itemize}
    \item if $\{a_1;a_2\}\subseteq [a_3]_{\rho}$ then $\{b_1;b_2\}\subseteq [b_3]_{\rho}$ and $m(a_1,a_2,a_3)=a_1\rho b_1=m(b_1,b_2,b_3)$
    \item otherwise, if there exist $1\leq i<j\leq 3$ such that $(a_i,a_j)\in\rho$, $l\notin\{i;j\}$ and $a_l\notin[a_i]_{\rho}$,  then
    $(b_i,b_j)\in\rho$, $l\notin\{i;j\}$ and $b_l\notin[b_i]_{\rho}$, hence $m(a_1,a_2,a_3)=a_i\rho b_i=m(b_1,b_2,b_3)$
    \item otherwise, if there exist $1\leq l<m\leq 3$ such that $(a_l,a_m)\in\theta$, then with transitivity of $\theta$ and the fact that $\rho\varsubsetneq \theta
    $ we have $(b_l,b_m)\in\theta$ and $m(a_1,a_2,a_3)=v_{[a_{l}]_{\theta}}=v_{[b_{l}]_{\theta}}=m(b_1,b_2,b_3)$
    \item otherwise, we have $m(a_1,a_2,a_3)=0=m(b_1,b_2,b_3)$.
\end{itemize}
Therefore $m\in \pol(\rho).$

Let $(a_i,b_i)\in\theta, 1\leq i\leq 3 $

\begin{itemize}
    \item if $\{a_1;a_2\}\subseteq [a_3]_{\rho}$ then with transitivity of $\theta$ and the fact that $\rho\varsubsetneq \theta
    $, we have $\{a_1;a_2\}\subseteq [a_3]_{\theta}$ and $\{b_1;b_2\}\subseteq [b_3]_{\theta}$; therefore
    $$(m(a_1,a_2,a_3), m(b_1,b_2,b_3))\subseteq \{a_i;b_i;v_{[a_{i}]_{\theta}};v_{[b_{l}]_{\theta}} \}^{2}\subseteq \theta $$
    \item otherwise, if there exist $1\leq i<j\leq 3$ such that $(a_i,a_j)\in\rho$, $l\notin\{i;j\}$ and $a_l\notin[a_i]_{\rho}$,  then
    $(b_i,b_j)\in\theta$; it follows that $m(b_1,b_2,b_3)\neq 0$ and $m(b_1,b_2,b_3)\in [m(a_1,a_2,a_3)]_{\theta}$. Since $l\notin\{i;j\}$ and $b_l\notin[b_i]_{\rho}$,
     $$m(a_1,a_2,a_3)=a_i\rho b_i=m(b_1,b_2,b_3)$$
    \item otherwise, if there exist $1\leq l<m\leq 3$ such that $(a_l,a_m)\in\theta$, then with transitivity of $\theta$, we have $(b_l,b_m)\in\theta$, and
    $$m(a_1,a_2,a_3)=v_{[a_{l}]_{\theta}}=v_{[b_{l}]_{\theta}}=m(b_1,b_2,b_3)$$
    \item otherwise, we have $m(a_1,a_2,a_3)=0=m(b_1,b_2,b_3)$.
\end{itemize}
Therefore $m\in \pol(\theta)$
   \item Assume that $\rho$ and $\theta$ satisfy condition $(b)$.
With the decomposition of $E_{k}$ into a cartesian product of
$E_{k}/\rho$ and $E_{k}/\theta $, we can say that, if $m_1$ is a
majority operation on $E_{k}/\rho$ and $m_2$ is a majority operation
on $E_{k}/\theta$, then the operation $m$ on $E_{k}/\rho\times
E_{k}/\theta $ that acts like $m_i$ in the $i$th coordinate($i=1,2$)
is a majority operation on $E_{k}/\rho\times E_{k}/\theta $ that
preserves $\rho$ and $\theta$.
\end{itemize}
\end{proof}

The two previous lemmas together with Theorem \ref{baker-pixley}
prove the sufficiency part of Theorem \ref{maintheo-equiv-equiv}.

Our Next step is to prove the necessity part of  Theorem
\ref{maintheo-equiv-equiv}. It is done in the following three
Lemmas.

\begin{lemma}\label{equivequivL1} If $\Delta_{E_{k}} \varsubsetneq \gamma \varsubsetneq \theta$ and
 $\gamma\varsubsetneq \rho$, then $D \varsubsetneq  \pol(\theta) \cap \pol(\gamma)
 \varsubsetneq  \pol(\theta).$
\end{lemma}

\begin{proof}
$ \gamma$ is a nontrivial equivalence relation on $E_{k}$ distinct from $\theta$ and $\rho$. Thus $D \subseteq
 \pol(\theta) \cap \pol(\gamma) \varsubsetneq  \pol(\theta)$. Let us prove that $D \neq  \pol(\theta) \cap \pol(\gamma)$;
 let $(a, b) \in \theta \setminus \gamma$
 and $(c, d) \in \rho\setminus\gamma.$ Then
$(a, b) \notin \rho$ and $(c, d) \notin \theta$. Define $f \in \mathcal{O}^{1}(E_{k})$ by:
$$f(x)=  \left\{\begin{array}{ll}
                                                 a &\text{ if } x\in B_{\mu(c)} \\
                                                 b &\text{ otherwise.}
                                               \end{array}
      \right.$$
Since $a\theta b$, $f \in \pol(\theta)$. In addition
 $\gamma\subseteq \theta$ and $f$ is constant on each block of $\theta$, hence
$f \in \pol(\gamma).$ Therefore $f \in  \pol(\theta) \cap \pol(\gamma)$
 while $f \notin D$ since $c\rho d$ and $(f(c), f(d)) = (a, b) \notin \rho.$
\end{proof}
\begin{lemma} Let $\rho$ and $\theta$ be two nontrivial equivalence relations which are incomparable.\\
If $\rho\cap\theta\neq\Delta_{E_{k}}$, then $\pol(\theta) \cap
\pol(\rho)\varsubsetneq \pol(\theta) \cap
\pol(\rho\cap\theta)\varsubsetneq\pol(\theta)$ and
 $\pol(\theta) \cap \pol(\rho\cap\theta)$ is maximal in $\pol(\theta)$.
\end{lemma}

\begin{proof}
It follows from the assumptions and Lemma \ref{equivequivL1} that $\pol(\theta) \cap \pol(\rho)\varsubsetneq \pol(\theta) \cap
\pol(\rho\cap\theta)\varsubsetneq\pol(\theta)$. As $\Delta_{E_{k}}\neq\rho\cap\theta\varsubsetneq\theta\neq \nabla_{E_{k}} $, the sufficiency part yields
$\pol(\theta) \cap \pol(\rho\cap\theta)$ is maximal in $\pol(\theta)$.
\end{proof}

\begin{lemma} Let $\rho$ and $\theta$ be two nontrivial equivalence relations which are incomparable.\\
If $\rho\cap\theta=\Delta_{E_{k}}$ and $\rho\circ\theta \neq
\nabla_{E_{k}}$, then for $
\sigma=\rho\circ\theta\cap\theta\circ\rho $, we have  $\pol(\theta)
\cap \pol(\rho)\varsubsetneq \pol(\theta) \cap
\pol(\sigma)\varsubsetneq\pol(\theta)$.
\end{lemma}
\begin{proof}
By assumptions we have $\sigma\neq \nabla_{E_{k}}$, hence
$\pol(\theta) \cap \pol(\sigma)\varsubsetneq\pol(\theta).$ From  the
definition of $\sigma$, we get $\pol(\theta) \cap
\pol(\rho)\subseteq \pol(\theta) \cap \pol(\sigma)$. Let us prove
that $\pol(\theta) \cap \pol(\rho)\varsubsetneq \pol(\theta) \cap
\pol(\sigma)$; choose $a\theta b$ with $a \neq b$ and $u\rho v$ with
$u \neq v$ and define the following unary operation $g$ on $E_{k}$
by:
$$g(x)=  \left\{\begin{array}{ll}
                                                 a &\text{ if } x=u \\
                                                 b &\text{ otherwise.}
                                               \end{array}
      \right.$$

As $(a,b)\in\theta$, we have $g\in \pol(\theta)$ and $g\in
\pol(\sigma)$; $(\theta\subseteq \sigma).$ Since
$(g(u),g(v))=(a,b)$, $(a,b)\in\theta\setminus\Delta_{E_{k}}$ and
$\rho\cap\theta=\Delta_{E_{k}}$, $g\notin \pol(\rho)$. Hence
$$g\notin\pol(\theta) \cap \pol(\rho)\text{ while }g\in\pol(\theta)
\cap \pol(\sigma).$$
\end{proof}
The two previous lemmas proved that if $\pol(\theta) \cap
\pol(\rho)$ is a submaximal clone of $\pol(\theta)$ and $\rho$ and
$\theta$ are incomparable, then $\rho\cap\theta=\Delta_{E_{k}}$ and
$\rho\circ\theta=\nabla_{E_{k}}$, so the condition $(b)$ holds.
\begin{proof}\textit{of Theorem} \ref{maintheo-equiv-equiv}
It follows from the previous lemmas.
\end{proof}
We conclude this section with the following theorem due to Lau and
Rosenberg, and characterizing the case of prime permutation
relations.
\begin{theorem}[\cite{GRE}]\label{maintheo-primepermut-equiv}
If $ \theta $  is a nontrivial equivalence relation and $ \rho $ is a graph of prime permutation $\pi$, on a finite set $E_{k} $, then $ \pol(\theta) \cap
\pol(\rho)$ is submaximal in $\pol(\theta)$ if and only if $\theta$ and $\rho$ satisfy one of the following statements:
\begin{enumerate}
    \item [(a)] $\rho \varsubsetneq \theta$
    \item [(b)] The image of an equivalence class of $\theta$ is include in another class of $\theta$ surjectively.
\end{enumerate}
\end{theorem}

\section{Central relations and $h$-regular relations}\label{sec5}
 In this section, $\theta$ is a nontrivial equivalence relation on $E_{k}$, whose equivalence classes are: $C_0,C_1, \ldots ,C_{t-1}$.
 Our aims is to characterize the central relations or $h$-regular relations $\rho$ such that $\text{Pol}(\rho)\cap\text{Pol}(\theta)$
 is maximal in $\text{Pol}(\theta) $.

Firstly, we give some definitions to be used. Let $\rho$ be an
$h$-ary relation $(h>1)$ on $E_{k}$. For $i\in\{0;1;2;\cdots;h-1\}$
we define  the relation $\rho_{i,\theta}$ by
\[\rho_{0,\theta}=\{(a_{1},\ldots,a_{h})\in E_{k}^{h}/ \exists u_{i}\in[a_{i}]_{\theta}, i\in\{1;\cdots;h\} \text{ with } (u_{1},u_{2},\ldots,u_{h})\in \rho\}\]
and for $i\geq 1$,
\begin{eqnarray*}\label{}
    \rho_{i,\theta}=\left\{(a_{1},\ldots,a_{h})\in E_{k}^{h}/ \exists u_{j}\in[a_{j}]_{\theta}, j\in\{i+1;\cdots;h\}\right.\\
\left.\qquad\text{ with } (a_{1},a_{2},\ldots,a_{i},u_{i+1},\ldots,u_{h})\in \rho\right\}.
\end{eqnarray*}

For $\sigma\in \mathcal{S}_{h}$ and $\gamma$ an $h$-ary relation, we
set $$\gamma_{\sigma}=\{(a_{\sigma(1)},\ldots,a_{\sigma(h)})/
(a_{1},\ldots,a_{h})\in\gamma\}.$$

For $J=\{j_1; \ldots; j_{n}\} \subseteq \{1; \ldots; h\}$ with $j_1< \ldots <j_{n}$, we define the $h$-ary relation $\rho_{J}$ or $\rho_{j_{1} \ldots
j_{n}}$ on $E_k$ as follow:
\begin{eqnarray*}
  \rho_{J} &=& \left\{\left( a_1,  a_2,\ldots,  a_h \right) | \exists u_{i.J}\in[a_{i}]_{\theta}, i\in \{1;\cdots;h\}\setminus J=\{i_1,...,i_{h-n}\}
\text{ such that }\right. \\
   & & \left.\left( a_{j_1},\ldots,  a_{j_n},  u_{i_1.J},\ldots,u_{i_{h-n}.J} \right)\in\rho\right\}.
\end{eqnarray*}

The next remark gives some properties of those relations.
\begin{remark}\text{ }
\begin{enumerate}
    \item For $i\in\{0;1;2;\cdots;h-1\}$, $\rho\subseteq\rho_{i,\theta}$ and $\pol(\rho)\cap\pol(\theta)\subseteq\pol(\rho_{i,\theta})$.
    \item $\rho_{0,\theta}$ is totally symmetric.
    \item For $\sigma\in \mathcal{S}_{h}$, $\pol(\rho)=\pol(\rho_{\sigma})$.
    \item If $\{j_1,...,j_n\}\subseteq\{r_1,...,r_m\}$, then $\rho_{\{r_1,...,r_m\}}\subseteq\rho_{\{j_1,...,j_n\}}.$
    \item for $J=\{1;\cdots;n\}$ we have $\rho_{J}=\rho_{n,\theta}$.
    \item For all $1\leq j_1< \ldots <j_{n}\leq h$, there exists a permutation $\sigma\in\mathcal{S}_{h}$ such that $\rho_{j_{1} \ldots
j_{n}}=(\rho_{n,\theta})_{\sigma}.$
\end{enumerate}
\end{remark}

\begin{definition}\label{def-theta-transversal} Let $\rho$ be an $h$-ary relation and $\theta $ be a nontrivial equivalence relation on $E_k$ with $t$ classes.

\begin{enumerate}
    \item There is a transversal $T$ for the $\theta $-classes means that there exist $u_1,\ldots,$ $u_t\in E_k$ such that $(u_i, u_j) \notin\theta $ for all $1 \leq i < j\leq t$,
$(u_{i_{1}},u_{i_{2}}, \ldots ,u_{i_{h}})\in \rho$ for all $1 \leq i_{1}, \ldots ,i_{h} \leq t$ and $T = \{u_1,  \ldots  , u_t\}$.
    \item There is a transversal $T$ of order $l$ ($1\leq l\leq h-1$) for the $\theta $-classes means that there exist
$u_1,\ldots,$ $u_t\in E_k$ such that $(u_i, u_j) \notin\theta $ for all $1 \leq i < j\leq t$,  $(a_1,a_2,\ldots,a_l,v_{l+1},v_{l+2},  \ldots  , v_{h})\in
\rho$,  for all    $a_1,a_2,...,a_l\in E_k$ and $v_{l+1},v_{l+2},  \ldots  , v_{h}\in\{u_{1};\cdots;u_{t}\}$; and $T=\{u_1,...,u_t\}$.
\end{enumerate}
\end{definition}
\begin{definition}
A transversal of order $0$  for the $\theta$-classes means a
transversal for the $\theta$-classes.
\end{definition}

\begin{definition}\label{def-theta-fermeture} Let $\rho$ be an $h$-ary relation and $\theta $ be a nontrivial equivalence relation on $E_k$ with $t$ classes.

\begin{enumerate}
    \item $\rho$ is $\theta$-closed means that $\rho=\rho_{0,\theta}$.
    \item $\rho$ is weakly $\theta $-closed of order $l$($1\leq l\leq h-1$)  means that there is a transversal $T=\{u_{1};\cdots;u_{t}\}$ of order $l-1$
    for the $\theta$-classes and  $\rho= \underset{\substack{\sigma\in\mathcal{S}_{h}}}{\bigcap} (\rho_{l,\theta})_{\sigma}.$
\end{enumerate}
\end{definition}

Secondly, we characterize some particular relations. We consider the
surjective map
$$\begin{array}{rccl}
                                  \varphi: & E_{k} & \rightarrow & E_{t} \\
                                   & x & \mapsto & \varphi(x)=i \text{ if } x\in C_{i}.
                                \end{array}
$$

For an $n$-ary relation $\alpha$ on $E_{t}$, we set
    \[\varphi^{-1}(\alpha)=\{(a_{1},\ldots,a_{n})\in E_{k}^{n}: (\varphi(a_{1}),\varphi(a_{2}),\ldots,\varphi(a_{n}))\in \alpha\};\]
for an $n$-ary relation $\beta$ on $E_{k}$, we set
\[\varphi(\beta)=\{(\varphi(a_{1}),\varphi(a_{2}),\ldots,\varphi(a_{n})): (a_{1},\ldots,a_{n})\in \alpha\}.\]
\begin{remark}
With the previous considerations and for a central relation $\rho$, we have:
\begin{itemize}
    \item $\rho$ is $\theta$-closed if and only if there exists an $h$-ary central relation $\gamma$ on $E_{t}$ such that $\rho=\varphi^{-1}(\gamma).$
    \item If $\rho$ is weakly $\theta $-closed of order $l$($1\leq l\leq h-1$),  then
    $\pol((\rho_{l,\theta})_{\sigma_{1}}\cap\cdots\cap(\rho_{l,\theta})_{\sigma_{n}})\subseteq\pol(\rho)$ for
    $\{\sigma_{1};\cdots;\sigma_{n}\}\subseteq\mathcal{S}_{h}$
\end{itemize}
\end{remark}

\begin{remark}\label{def-theta-close-trans-binary}
 Let  $\rho$  be a binary relation and   $\theta$   be a nontrivial equivalence relation on $E_k$ with $t$ classes.
\begin{enumerate}
    \item [(1)]  $\rho$  is   $\theta$-closed if and only if  $\rho  =  \theta\circ
\rho\circ \theta$,
    \item [(2)] $\rho$ is weakly $\theta$-closed of order $1$(or simply weakly $\theta$-closed) if and only if $\rho  = (
    \theta\circ\rho)\cap(\rho\circ\theta),$ and there is a transversal $T$ for the $\theta$-classes.
\end{enumerate}
\end{remark}

Thirdly, we characterize the central relations $\rho$ generating the
submaximality.

\subsection{Central relations}\label{central-relation-subsection}
We recall that for $k=2$, we have the result in the Post's description. If $k\geq 3$ and $h\in\{1;2\}$ the following results give the characterization of
existing submaximal classes.
\begin{lemma}\cite{LAU}
If $h=1$, then $\text{Pol}(\rho)\cap\text{Pol}(\theta)$ is maximal in $\text{Pol}(\theta) $ if and only if the following condition is valid:
$$(\exists I\subset\{0;\cdots;t-1\}: \rho=\underset{i\in I}\cup C_{i})\vee \forall j\in\{0;\cdots;t-1\}, \rho\cap C_{j}\neq \emptyset  $$
\end{lemma}

\begin{theorem}\cite{TEM-ROS}
If $h=2$, then $\text{Pol}(\rho)\cap\text{Pol}(\theta)$ is maximal in $\text{Pol}(\theta) $ if and only if one of the following conditions is valid:
\begin{enumerate}
    \item [(i)] $\theta\subseteq\rho$ and every $\theta$-class contains a central element of $\rho$;
    \item [(ii)] $\rho$ is $\theta$-closed;
    \item [(iii)] $\rho$ is weakly $\theta$-closed of order $1$.
\end{enumerate}
\end{theorem}

In the remaining of this subsection we suppose that $h\geq 3$ and $k\geq 3$.

For a nontrivial equivalence relation $\theta$, we define the $h$-ary relation $\eta$ by
$$\eta=\left\{\left(u_1 , u_2 ,\ldots,u_h\right)\in E_{k}^{h}
/ u_1\theta u_2\right\}.$$ Here we state the main theorem of this subsection:
\begin{theorem}\label{maintheorem}
Let $k \geq 3$ and let $\theta $ be a nontrivial equivalence
relation on $E_k$ with $t$ equivalence classes. For an $h$-ary
central relation $\rho $ on $E_k$, the clone
$\text{Pol}(\rho)\cap\text{Pol}(\theta)$  is a  submaximal clone  of
$\text{Pol}(\theta) $ if and only if $h\leq t $ and  $\rho $
satisfies one of the following three conditions:
\begin{enumerate}
    \item [I.] $\eta\subseteq \rho$ and every $\theta $-class contains a central element of $\rho $;
    \item [II.] $\rho $ is $\theta $-closed;
    \item [III.] $\rho$ is weakly $\theta $-closed of order $l$ and $\eta\subseteq \rho$.
\end{enumerate}
\end{theorem}
The proof of Theorem \ref{maintheorem} will follow from the results
obtained below. It will be given at the end of this subsection.
\begin{definition}
Let $l\in\{I; II; III\}$. $\rho $ is of type $l$ if $\rho $ satisfies the condition $l$ of Theorem \ref{maintheorem}.
\end{definition}
The following examples clarify the type of relations defined above.
\begin{example}
Let $k\geq 3$ be an integer and $0\leq i<j<r<n\leq k-1$, we denote
by $A_{i,j,r}$ and $A_{i,j,r,n}$ the sets

$$A_{i,j,r}:=\{(\sigma(i),\sigma(j),\sigma(r)); \sigma\in\mathcal{S}_{\{i;j;r\}}\} $$ and
$$A_{i,j,r,n}:=\{(\sigma(i),\sigma(j),\sigma(r),\sigma(n)); \sigma\in\mathcal{S}_{\{i;j;r;n\}}\}.$$

We consider the following equivalence relations $\theta_{i}$ defined by their equivalence classes denoted by $C_{m}^{i}$

$\theta_{1} \text{ is defined on } E_{6} \text{ by } C_{0}^{1}=\{0;1\}, C_{1}^{1}=\{2;3\}, C_{2}^{1}=\{4;5\} $;

$\theta_{2} \text{ is defined on } E_{5} \text{ by } C_{0}^{2}=\{0;1\}, C_{1}^{2}=\{2\}, C_{2}^{2}=\{3\}, C_{3}^{2}=\{4\}$;

$\theta_{3} \text{ is defined on } E_{4} \text{ by } C_{0}^{3}=\{0;1\}, C_{1}^{3}=\{2\}, C_{2}^{3}=\{3\}$;

$\theta_{4} \text{ is defined on } E_{8} \text{ by } C_{0}^{4}=\{0;1;2\}, C_{1}^{4}=\{3;4;5\}, C_{2}^{4}=\{6;7\}$;

$\theta_{5} \text{ is defined on } E_{8} \text{ by } C_{0}^{5}=\{0;1;2\}, C_{1}^{5}=\{3;4\}, C_{2}^{5}=\{5;6\}, C_{3}^{5}=\{7\}$;

and the relations

$$\Upsilon_{1}=E_{6}^{3}\setminus A_{1,2,5},\Upsilon_{2}=E_{5}^{3}\setminus A_{2,3,4},\Upsilon_{3}=E_{4}^{3}\setminus A_{1,2,3}$$

$$\Upsilon_{4}=E_{8}^{3}\setminus (A_{1,4,6}\cup A_{1,4,7}\cup A_{1,5,7}\cup A_{1,5,6}\cup A_{2,3,7}\cup A_{2,4,6}\cup A_{2,4,7}\cup A_{2,5,6}\cup A_{2,5,7} )$$

$$\Upsilon_{5}=E_{8}^{4}\setminus (A_{1,4,6,7}\cup A_{2,4,6,7})$$

It is easy to see that: $\Upsilon_{1}$ is a central relation of type $I$ with $\theta_{1}$; $\Upsilon_{2}$ is a central relation of type $II$ but not of
type $I$ with $\theta_{2}$; with $\theta_{3}$, $\Upsilon_{3}$ is a central relation whose center is $\{0\}$   but it is neither of type, $I$, $II$ or
$III$.

$\Upsilon_{4}$ is weakly $\theta_{4} $-closed of order $2$ with a
transversal of order $1$, $T_{1}=\{0;3;6\}$

$\Upsilon_{5}$ is weakly $\theta_{5} $-closed of order $3$ with a
transversal of order $2$, $T_{2}=\{0;3;5;7\}$
\end{example}

\begin{definition}\label{def-rel-diagonale}
Let $\theta$ be an equivalence relation on $E_k$. An $h$-ary
relation $\tau$  on $E_k$ is said to be diagonal through $\theta$ if
there exists an equivalence relation $\varepsilon_{1}$ on $\{1;2;
\ldots ;h\}$ with equivalence classes $A_1,A_2, \ldots ,A_l$ and an
equivalence relation $\varepsilon_{2}$ on $\{\min (A_m); 1\leq m\leq
l\}$ such that $$\tau=\left\{\left(a_{1},a_{2},\hdots, a_{h}
\right)\in E_k^{h}/ ((i,j)\in\varepsilon_{1}\Rightarrow
a_{i}=a_{j})\text{ and }((i,j)\in\varepsilon_{2}\Rightarrow
a_{i}\theta a_{j})\right\}.$$
\end{definition}

Given two equivalence relations $\theta_1$ and $\theta_2$ satisfying
Definition \ref{def-rel-diagonale}, we denote by
$D_{\theta_1\theta_2}$ the corresponding diagonal relation through
$\theta$.
\subsubsection{Proof of the necessity criterion in Theorem \ref{maintheorem}}\label{sec5'-subsec5}
\begin{proposition}\label{sufficiency-direction}
If $k \geq 3$, $\theta $ is a nontrivial equivalence relation on
$E_k$, and $\rho$ is an $h$-ary central relation  on $E_k$ such that
one of conditions I-III is satisfied, then the clone
$\text{Pol}(\rho)\cap\text{Pol}(\theta)$  is a submaximal clone of
$\text{Pol}(\theta) $.
\end{proposition}
Before the proof of Proposition \ref{sufficiency-direction}, we give
some results characterizing relations containing
$\text{Pol}(\rho)\cap\text{Pol}(\theta)$ and we show that
$\text{Pol}(\rho)\cap\text{Pol}(\theta)$ contains an
$h$-near-unanimity operation.

For the proof of this proposition we choose a fixed central element
$c$ of $\rho $; we denote by $C_\rho$ the set of all central
elements of $\rho$. If $\rho $ is of type I, choose a central
element $c_B$ from the $\theta$-class $B$ which can be $\min(B\cap
C_\rho)$, and if $\rho $ is of type III and of order $l$, choose a
transversal $T$ of order $l-1$ of the $\theta$-classes and denote by
$T_B$ the element of $T$ representing the $\theta$-class $B$.

We begin with the following lemma characterizing the diagonal
relations through $\theta$.
\begin{lemma}\label{caracterisation-I}
For an equivalence relation $\theta$ on $E_k$ and a diagonal
relation $\tau$ through $\theta$, with arity $h$ on $E_k$, we have
$\pol(\tau)=\pol(\theta)$ or $\pol(\tau)=\mathcal{O}(E_k)$.
\end{lemma}
\begin{proof} Let $\theta$ be an equivalence relation on $E_k$ and $\tau=D_{\varepsilon_{1}\varepsilon_{2}}$ be
a diagonal relation through $\theta$. Let $T=\{\min A_m; 1\leq m\leq l\}$ where $A_{m}, 1\leq m\leq l$ are as in Definition $\ref{def-rel-diagonale}$. We
will distinguish two cases: (a) $\varepsilon_{2}\neq \Delta_{T}$ and (b) $\varepsilon_{2}= \Delta_{T}.$
\begin{enumerate}
    \item [(a)]Assume that $\varepsilon_{2}\neq \Delta_{T}.$\\ There exist $u,v\in T$ with $u<v$ such that for all
$(a_1,a_2,\hdots,a_h)\in\tau$, we have $(a_u,a_v)\in\theta$. Using the definition of $\tau$, we have $\pol(\theta)\subseteq \pol(\tau).$ By setting
$$pr_{uv}(\tau):=\{(e_{u}^{h}(\textbf{a}),e_{v}^{h}(\textbf{a})); \textbf{a}=(a_1,a_2,\ldots,a_h)\in\tau\},$$ it follows that $pr_{uv}(\tau)=\theta$, therefore $\pol(\tau)\subseteq
\pol(\theta)$, and it appears that $\pol(\tau)=\pol(\theta)$.
    \item [(b)]Assume that $\varepsilon_{2}= \Delta_{T}.$\\ It is clear that
    $$\tau=D_{\varepsilon_{1}\Delta_{T}}=\left\{\left(a_{1},a_{2},\hdots, a_{h} \right)\in E_k^{h}/ i\varepsilon_{1}j\Rightarrow a_{i}=a_{j}\right\}.$$
    Hence
 $\pol(\tau)=\mathcal{O}(E_k).$
\end{enumerate}
\end{proof}
\begin{lemma}\label{caracterisation-diagonal-all}
Under the assumptions of Proposition \ref{sufficiency-direction}, we have:
\begin{enumerate}
    \item [(a)] $\eta\subseteq \rho$;
    \item [(b)] If $\rho $ is of type I or II, then an $h$-ary relation $\tau $ on $E_k$ is preserved by every operation in $\text{Pol}(\rho)\cap\text{Pol}(\theta)$
    if and only if $\tau $ is either the empty relation, a diagonal
relation through $\theta$,  or the relation $\rho$;
    \item [(c)] If $\rho $ is of type III and of order $l$, then an $h$-ary relation $\tau $ on $E_k$ is preserved by  every operation in $\text{Pol}(\rho)\cap\text{Pol}(\theta)$
    if and only if $\tau $ is either the empty relation, a diagonal
relation through $\theta$, or an intersection of relations of the
form $(\rho_{l,\theta})_{\sigma}$ with $\sigma\in\mathcal{S}_{h}$.
\end{enumerate}
\end{lemma}
\begin{proof}
\begin{enumerate}
    \item [(a)] The proof of $\eta\subseteq \rho$ is straightforward.
    \item [(b)]  Assume $\tau $ is either the empty relation, or a diagonal
relation through $\theta$. Then $\pol(\tau)\in\{\pol(\theta);\mathcal{O}(E_k)\}$ (see Lemma \ref{caracterisation-I}); hence $
\pol(\rho)\cap\pol(\theta)\subseteq\pol(\tau)$. Assume $\tau$ is of the form $(\rho_{l,\theta})_{\sigma}$. Then it is easy to see that $
\pol(\rho)\cap\pol(\theta)\subseteq\pol((\rho_{l,\theta})_{\sigma})$. Hence, if $\tau$ is an intersection of relations of the form
$(\rho_{l,\theta})_{\sigma}$ then $\pol(\rho)\cap\pol(\theta)\subseteq\pol(\tau)$.

Conversely, assume $\tau$ is preserved by all operations in $\pol(\rho)\cap\pol(\theta)$.

Let us suppose that $\tau$ is not the empty relation. We have to prove that $\tau$ is  either a diagonal relation through $\theta$, or an intersection of
relations of the form $(\rho_{l,\theta})_{\sigma}$ with $l$ the order of $\rho$ and $\sigma\in\mathcal{S}_{h}$. \\
For this purpose, we define two equivalence relations. The first one denoted by $\epsilon_{1}$ is defined on $\{1;2; \ldots ;h\}$ by:
$$i\epsilon_{1}j\text{ iff }\forall \left(a_1,  a_2, \ldots, a_h\right)\in \tau, a_i=a_j.$$ $\epsilon_{1}$ is an equivalence relation with classes $A_{0}, A_{1},\ldots,
A_{m}$. The second one, denoted by  $\epsilon_{2}$ is defined on $T=\{\min(A_i); 0\leq i\leq m\}$ by:  $$i\epsilon_{2}j \text{ iff }\forall \left(a_1, a_2,
\ldots, a_h\right)\in \tau,(a_i,a_j)\in\theta.$$ It follows that $D_{\epsilon_{1}\epsilon_{2}}$ is a diagonal relation through $\theta$. In order to
complete the proof of this lemma, we distinguish two cases: $(\epsilon_{1}\neq\Delta_{\{1;2; \ldots ;h\}} \text{ or } \epsilon_{2}\neq\Delta_{T})$ and
$(\epsilon_{1}=\Delta_{\{1;2; \ldots ;h\}} \text{ and } \epsilon_{2} = \Delta_{T})$.
\begin{enumerate}
    \item [(i)] We suppose that $\epsilon_{1}\neq\Delta_{\{1;2; \ldots ;h\}} \text{ or } \epsilon_{2}\neq\Delta_{T}.$ We will prove that
$D_{\epsilon_{1}\epsilon_{2}}=\tau$. It suffices to show that $D_{\epsilon_{1}\epsilon_{2}}\subseteq\tau$; $\tau\subseteq D_{\epsilon_{1}\epsilon_{2}}$ by
definitions of $\epsilon_{1}$ and $\epsilon_{2}$.  We need only consider three subcases: $(\epsilon_{1}=\nabla_{\{1;2; \ldots ;h\}})$,
$(\epsilon_{1}\neq\nabla_{\{1;2; \ldots ;h\}} \text{ and } \epsilon_{2}=\nabla_{T})$, and $(\epsilon_{1}\neq\nabla_{\{1;2; \ldots ;h\}} \text{ and }
\epsilon_{2}\neq\nabla_{T})$.
\begin{enumerate}
    \item [a)] If $\epsilon_{1}=\nabla_{\{1;2; \ldots ;h\}},$ then $D_{\epsilon_{1}\epsilon_{2}}=\{(x,...,x); x\in E_{k}\}$. Since
    $\tau\subseteq D_{\epsilon_{1}\epsilon_{2}}$ and $\tau$ is not the empty relation, for each $b\in E_k$ the constant function of value $b$
    preserves $\theta$ and $\rho$; hence $(b,...,b)\in \tau$ and $\tau=D_{\epsilon_{1}\epsilon_{2}}.$
    \item [b)] If $\epsilon_{1}\neq\nabla_{\{1; \ldots ;h\}} \text{ and }
\epsilon_{2}=\nabla_{T}$, then there exists $(i,j)\in \nabla_{\{1; \ldots ;h\}}$ such that $(i,j)\notin\epsilon_{1}.$ For all $1\leq i,j\leq h$ such that
$(i,j)\notin\epsilon_{1}$, we choose $\mathbf{a}^{ij}=(a^{ij}_{1},...,a^{ij}_{h})\in\tau$ such that $a^{ij}_{i}\neq a^{ij}_{j}$ and we set
$B=\{\mathbf{a}^{ij}; (i,j)\notin \epsilon_{1}\}$. Let $q=|B|$; to simplify our notation, we suppose that $B=\{\mathbf{b}^1;\mathbf{b}^2; \ldots
;\mathbf{b}^q\}$ and we define the sequence $(\mathbf{x}_i)_{1\leq i\leq h}$ by $\mathbf{x}_i=(b^1_{i};b^2_{i}; \ldots ;b^q_{i})$. It is easy to see that:
$(i,j)\notin \epsilon_{1}\Rightarrow \mathbf{x}_i\neq \mathbf{x}_j$. Let $(u_1,...u_h)\in D_{\epsilon_{1}\epsilon_{2}}$, we consider the $q$-ary operation
defined on $E_k$ by:
$$f(\mathbf{y})=  \left\{\begin{array}{ll}
                                                 u_l &\text{ if } \mathbf{y}=\mathbf{x}_l\in\{\mathbf{x}_1;\mathbf{x}_2; \ldots ;\mathbf{x}_h\} \\
                                                 u_1 &\text{ otherwise }.
                                               \end{array}
      \right.
 $$ Since $\{u_1,u_2,...u_h\}^{2}\subseteq \theta$ and $\,\eta\subseteq \rho$, $f\in\pol(\theta)\cap\pol(\rho)$. Therefore
$f\in\pol(\tau)$ and $(u_1,...,u_h)\in\tau$.
    \item [c)] If $\epsilon_{1}\neq\nabla_{\{1;2; \ldots ;h\}} \text{ and } \epsilon_{2}\neq\nabla_{T}$, then there exists
    $(i,j)\in\nabla_{T}$ such that $(i,j)\notin\epsilon_{2}$.
     For all $1\leq i,j\leq h$  such that $(i,j)\notin
\epsilon_{2}$  we choose  $ \mathbf{a}^{ij}=\left(a^{ij}_1,  a^{ij}_2, \ldots, a^{ij}_h \right)\in \tau$ such that $(a^{ij}_i,a^{ij}_j)\notin \theta$ and
we consider the set $B=\{\mathbf{a}^{ij}, (i,j)\notin \epsilon_{2}\}$. We set $q=|B|$. To simplify our notation  we set $B=\{\mathbf{b}^1;\mathbf{b}^2;
\ldots ;\mathbf{b}^q\}$ which allows us to define  $\mathbf{x}_1=(b^1_1,b^2_1, \ldots ,b^q_1)$, $\mathbf{x}_2=(b^1_2,b^2_2, \ldots ,b^q_2)$, \ldots ,
$\mathbf{x}_s=(b^1_s,b^2_s, \ldots ,b^q_s)$. We remark that $(i,j)\notin \epsilon_{2}\Rightarrow (\mathbf{x}_i,\mathbf{x}_j)\notin \theta$.\\
Let $\left(u_1, u_2, \ldots,  u_h\right)\in D_{\epsilon_{1}\epsilon_{2}}$  and consider the $q$-ary operation defined on $E_k$  by:

$$ f(\mathbf{y})=  \left\{\begin{array}{ll}
                                                 u_l &\text{ if } \mathbf{y}=\mathbf{x}_l \\
                                                 u_{\sigma(\mathbf{y})} &\text{ if } \mathbf{y}\notin\{\mathbf{x}_1;\mathbf{x}_2; \ldots ;\mathbf{x}_h\}, \exists l\in\{1,2, \ldots ,h\}/ \mathbf{y}\theta \mathbf{x}_l \text{ and}\\%
                                                 &\,\,\,\,\,\,\,\,\,\,\,\,\sigma(\mathbf{y})=min\{t / \mathbf{y} \theta \mathbf{x}_{t}\} \\
                                                 c &\text{ otherwise, }
                                               \end{array}
      \right.
$$ where $c$ is a central element. Since $\rho$ is totally reflexive and $D_{\epsilon_{1}\epsilon_{2}}\subseteq \rho$,  it follows that
$f\in \pol(\rho)$. By  the definition of $f$, we have $f\in \pol(\theta)$, thus $f\in \pol(\tau)$. Since $\mathbf{b}^1,..,\mathbf{b}^q\in \tau$, it follows
that $f(\mathbf{b}^1,\mathbf{b}^2, \ldots ,\mathbf{b}^q)=\left(f(\mathbf{x}_1), f(\mathbf{x}_2), \ldots, f(\mathbf{x}_h)\right)=\left(u_1,  u_2, \ldots,
u_h\right)\in \tau$, therefore $D_{\epsilon_{1}\epsilon_{2}}\subseteq \tau.$
\end{enumerate}
    \item [(ii)] We suppose that $\epsilon_{1}=\Delta_{\{1;2; \ldots ;h\}} \text{ and }
\epsilon_{2} = \Delta_{T}$.  We show that $\tau$ is  an intersection of relations of the form $(\rho_{l,\theta})_{\sigma}$ or the relation
$E_{k}^{h}$.\\
   For all $1\leq i,j\leq h$  such that  $i \neq j$, choose  $\mathbf{a}^{ij}=\left(a^{ij}_1,\ldots, a^{ij}_s\right)\in\tau$ such that
   $(a^{ij}_i,a^{ij}_j)\notin \theta$ and consider the set $B=\{\mathbf{a}^{ij}, (i,j)\notin \epsilon_{2}\}$. Using similar notation as in part (c) of (i) and the same
   $q$-ary operation  $f$  on  $E_k$ for a given $\left(u_1, u_2, \ldots, u_h\right)\in \rho $, we have
$$f(\mathbf{b}^1,\mathbf{b}^2, \ldots ,\mathbf{b}^q)=\left(f(\mathbf{x}_1), f(\mathbf{x}_2), \ldots, f(\mathbf{x}_h)\right)=\left(u_1, u_2,\ldots, u_h\right)\in \tau ,$$ and then $\rho\subseteq \tau$.
By the definition of $(\rho_{l,\theta})_{\sigma}$ with $l$ the order of $\rho$ and $\sigma\in\mathcal{S}_{h}$, we have: $\rho\subseteq
(\rho_{l,\theta})_{\sigma}$ for all $\sigma\in\mathcal{S}_{h}.$
 We distinguish once more two subcases: $\rho=\tau$ or $\rho\neq\tau$.
\begin{enumerate}
    \item [a)]If $\rho=\tau$, then it is finished.
    \item [b)]Otherwise $\rho\varsubsetneq\tau\subseteq E_{k}^{h}.$

If $\rho$ is of type I or II, then we will show that $\tau= E_{k}^{h}$. \\
As $\rho\varsubsetneq\tau$ then $\tau\setminus\rho\neq\emptyset$. Let us consider $\left(a_1,a_2,\ldots,a_h\right)\in\tau\setminus\rho$. Let $\left(u_1,
u_2, \ldots, u_h\right)\in E_{k}^{h}$; assume $\rho$ is of type I and consider the unary operation $f$  defined on  $E_k$  by:$$f(x)=
\left\{\begin{array}{l}
                                                 u_i \text{ if } x= a_i \\
                                                 c_{[a_i]_{\theta}} \text{ if } x\theta a_i \text{ and } x\neq a_i\ \\
                                                 c \text{ otherwise,}
                                               \end{array}
      \right. $$ where $c$ is a central element.
$f\in \pol(\rho)$ because $ \eta\subseteq\rho$ and $\rho$ is totally symmetric. Moreover $f\in \pol(\theta)$, then $f\in \pol(\tau)$
 and  $$\left(u_1, u_2, \ldots, u_h\right)=\left(f(a_1), f(a_2), \ldots, f(a_h)\right)\in\tau.$$
Assume $\rho$ is of type II and consider the unary operation $f$  defined on  $E_k$  by:$$f(x)=  \left\{\begin{array}{l}
                                                 u_i \text{ if } x\theta a_i \\
                                                 c \text{ otherwise }
                                               \end{array}
      \right. $$ where $c$ is a central element.
$f\in \pol(\rho)$ because $\rho$ is $\theta$-closed. Moreover $f\in \pol(\theta)$, then $f\in \pol(\tau)$
 and $\left(u_1, u_2, \ldots, u_h\right)=\left(f(a_1), f(a_2), \ldots, f(a_h)\right)\in\tau$. Thus $\tau=E_{k}^{h}=D_{\Delta_{\{1;\cdots;h\}}\Delta_{\{1;\cdots;h\}}}.$

Let's suppose that $\rho$ is of type III and of order $l$, i.e.,
$\rho$
is weakly $\theta$-closed of order $l$ and $\eta\subseteq\rho$.\\
Since $\tau\setminus\rho$ is not empty, there exists $\left(a_1,
a_2, \ldots, a_h\right)\in(\tau\setminus\rho).$ Therefore
$\left(a_1, a_2, \ldots,
a_h\right)\notin\underset{\sigma\in\mathcal{S}_{h}}\bigcap(\rho_{l,\theta})_{\sigma}$.
We suppose that there is $\sigma'\in\mathcal{S}_{h}$ such that
$(a_1,a_2,\ldots,a_h)\in(\rho_{l,\theta})_{\sigma'}.$ We consider
the set $R_{1}:=\{\sigma'\in\mathcal{S}_{h}/
\,\,(a_1,a_2,\ldots,a_h)\in(\rho_{l,\theta})_{\sigma'}\}$ and define
the relation $\varphi$ by
$$\varphi=\underset{\sigma'\in R_1}\bigcap(\rho_{l,\theta})_{\sigma'}.$$ We have $\rho\varsubsetneq \varphi$.

We will show that $\varphi\subseteq \tau$. Let $(u_1,...u_h)\in\varphi$ and set
\begin{eqnarray*}
  D=\{(b_1,\ldots,b_h)\in E_{k}^{h}; b_i\in\{u_1,...,u_h\}\cup\{c,T_{[u_{1}]_{\theta}},\dots, T_{[u_{h}]_{\theta}}\}, 1\leq i\leq h, \\
   \text{ and } \{b_1,...b_s\}\neq\{u_1,...,u_h\}\}.
\end{eqnarray*}
We define the unary operation $h$ on $E_k$ by:
$$h(x)=  \left\{\begin{array}{l}
                                                 u_i \text{ if } x= a_i \\
                                                 T_{[u_i]_{\theta}} \text{ if } x\in [a_i]_{\theta}\setminus\{a_i\} \\
                                                 c \text{ otherwise }
                                               \end{array}
      \right. $$
For $(y_1,...,y_h)\in\rho$, we have $(h(y_1),...,h(y_h))\in D\cap\rho\subseteq\rho$. Therefore $h\in \pol(\rho)\cap\pol(\theta)$. Hence $h\in \pol(\tau)$
and $(u_1,...,u_h)=(h(a_1),...,h(a_h))\in\tau.$\\
If $\tau=\varphi$, then it is finished. Otherwise, we have $\varphi\varsubsetneq\tau$.

There exists $(a_1,a_2,\ldots,a_h)\in \tau\setminus\varphi$. If
there exists $s\in\mathcal{S}_{h}$ such that
$(a_1,a_2,\ldots,a_h)\in (\rho_{l,\theta})_{s}$, then we use the
same argument to construct $\varphi'$ such that
$\varphi\varsubsetneq\varphi'$ and $\varphi'\subseteq \tau$.
Therefore $\tau=\varphi'$ or $\varphi'\varsubsetneq\tau$. So we have
the same conclusion as above. We continue this process until
obtained a $h$-tuple $(a_1, a_2, \ldots, a_h)\in\tau$ such that, for
each $\sigma\in\mathcal{S}_{h}$,  $(a_1, a_2, \ldots, a_h)\notin
(\rho_{l,\theta})_{\sigma}$. Let $(u_1,...,u_h)\in E_{k}^{h}$, using
the unary operation $h$ defined above and the fact that $\rho$ is
weakly $\theta$-closed of order $l$ and there is a transversal $T$
of order $l-1$ for the $\theta$-classes, we show that
$(u_1,...,u_h)\in \tau.$ Therefore $\tau=E_{k}^{h}$.
\end{enumerate}
\end{enumerate}
\end{enumerate}
\end{proof}
\begin{lemma}\label{existence-unanimity-all}
Under the assumptions of Proposition \ref{sufficiency-direction}, $\pol(\theta)\cap \pol(\rho)$ contains an $h$-near-unanimity operation.
\end{lemma}

\begin{proof}
If $\rho$ is of type I,  let us consider the $(s+1)$-ary operation $m$ defined on $E_k$ by:
$$ m(x_1, \ldots ,x_h,x_{h+1})=  \left\{\begin{array}{ll}
                                                 x_{i_1}& \text{ if there exist }  1\leq i_1 < \ldots <i_h\leq h+1 \\ & \,\, \text{ such that } x_{i_1}=x_{i_2}= \ldots =x_{i_s} \\
                                                 c_{[x_{i_1}]_{\theta}}& \text{ if there exist }  1\leq i_1 < \ldots <i_h\leq h+1 \\ & \,\, \text{ such that }  [x_{i_1}]_{\theta}=[x_{i_2}]_{\theta}= \ldots =[x_{i_s}]_{\theta} \\
                                                 c & \text{ otherwise;}
                                               \end{array}
      \right.$$

if $\rho$ is of type II,  let us consider the $(s+1)$-ary operation $m$ defined on $E_k$ by:
$$ m(x_1, \ldots ,x_h,x_{h+1})=  \left\{\begin{array}{ll}
                                                 x_{i_1}& \text{ if there exist }  1\leq i_1 < \ldots <i_h\leq h+1 \\ & \,\, \text{ such that } x_{i_1}=x_{i_2}= \ldots =x_{i_s} \\
                                                 x_{i_1}& \text{ if there exist }  1\leq i_1 < \ldots <i_h\leq h+1 \\ & \,\, \text{ such that }  [x_{i_1}]_{\theta}=[x_{i_2}]_{\theta}= \ldots =[x_{i_h}]_{\theta} \\
                                                 c & \text{ otherwise;}
                                               \end{array}
      \right.$$

if $\rho$ is of type III,  let us consider the $(s+1)$-ary operation $m$ defined on $E_k$ by:
$$ m(x_1, \ldots ,x_h,x_{h+1})=  \left\{\begin{array}{ll}
                                                 x_{i_1}& \text{ if there exist }  1\leq i_1 < \ldots <i_h\leq h+1 \\ & \,\, \text{ such that } x_{i_1}=x_{i_2}= \ldots =x_{i_h} \\
                                                 T_{[x_{i_1}]_{\theta}}& \text{ if there exist }  1\leq i_1 < \ldots <i_h\leq h+1 \\ & \,\, \text{ such that }  [x_{i_1}]_{\theta}=[x_{i_2}]_{\theta}= \ldots =[x_{i_h}]_{\theta} \\
                                                 c & \text{ otherwise;}
                                               \end{array}
      \right.$$ where $c$ is a central element.
By definition, $m$ is a near unanimity function of order $h+1$. We will show that $m\in \pol(\rho)\cap \pol(\theta)$. To do this, we will prove firstly
that $m\in \pol(\theta)$ and secondly that $m\in \pol(\rho).$
\begin{enumerate}
    \item [a.] To see that $m$ preserves $\theta$, assume that $\left( u_i ,v_i \right)\in\theta$ for $1\leq i \leq s+1.$ Since
$$|\{[u_{1}]_{\theta};[u_2]_{\theta}; \ldots ;[u_{h+1}]_{\theta}\}|=|\{[v_1]_{\theta};[v_2]_{\theta}; \ldots ;[v_{h+1}]_{\theta}\}|,$$ if there exist $1\leq  i_1< i_2<..<i_h\leq h+1$ such that
$[u_{i_1}]_{\theta}=[u_{i_2}]_{\theta}= \ldots =[u_{i_h}]_{\theta}$, then $m(u_1,u_2, \ldots ,u_{h+1})\in [u_{i_1}]_{\theta}= [v_{i_1}]_{\theta}$ because $[u_{i}]_{\theta}=[v_{i}]_{\theta}$ for $i = 1, 2,..,h+1$.\\
Therefore, $(m(u_1,u_2, \ldots ,u_{h+1}),m(v_1,v_2, \ldots ,v_{h+1}))\in [u_{i_1}]_{\theta}^{2}\subseteq\theta$. In the other case we have $(m(u_1,u_2,
\ldots ,u_{h+1}),m(v_1,v_2, \ldots ,v_{h+1}))=\left( c ,c \right)\in\theta$; then $m\in \pol(\theta)$.
    \item [b.] Let us prove now that $m$ preserves $\rho$.\\
Let $\left(u_{11}, u_{21},\ldots, u_{h1}\right)$,$\left(u_{12}, u_{22}, \ldots, u_{h2}\right)$,\ldots,$\left(u_{1 h+1}, u_{2 h+1}, \ldots, u_{h
h+1}\right)$ $\in\rho.$ \\ If $\{m(u_{i1},u_{i2}, \ldots ,u_{i h+1}); \ 1\leq i \leq h\}$ contains a central element of $\rho$  then
$$\left(m(u_{11},u_{12}, \ldots ,u_{1 h+1}), \ldots, m(u_{h1},u_{h2}, \ldots ,u_{h h+1})\right)\in\rho;$$
else we will distinguish the three type of $\rho$.\\
Suppose $\rho$ be of type I, there exist $1\leq i_{r}^{1}<i_{r}^{2}< \ldots <i_{r}^{h}\leq h+1, \ 1\leq r \leq h$ such that $u_{r i_{r}^{1}}=u_{r
i_{r}^{2}}= \ldots =u_{r i_{r}^{h}}$ for $r\in\{1;\cdots;h\}$; as $\{i_{1}^{1};i_{1}^{2}; \ldots ;i_{1}^{h}\}\cap\{i_{2}^{1};i_{2}^{2}; \ldots
;i_{2}^{h}\}\cap, \ldots ,\cap\{i_{h}^{1};i_{h}^{2}; \ldots ;i_{h}^{h}\}\neq \emptyset$, let us consider a fixed element
 $i\in\{i_{1}^{1};i_{1}^{2}; \ldots ;i_{1}^{h}\}\cap\{i_{2}^{1};i_{2}^{2}; \ldots ;i_{2}^{h}\}\cap, \ldots ,\cap\{i_{h}^{1};i_{h}^{2}; \ldots ;i_{h}^{h}\}$.
 We have
 $$( m(u_{11},u_{12}, \ldots ,u_{1 h+1}), \ldots ,m(u_{h1},u_{h2}, \ldots ,u_{h h+1}))=(u_{1i},...,u_{hi})\in\rho.$$
 Therefore $m\in \pol(\rho)$.\\
If $\rho$ is of type II there exist $1\leq i_{r}^{1}<i_{r}^{2}< \ldots <i_{r}^{h}\leq h+1, \ 1\leq r \leq h$ such that $[u_{r i_{r}^{1}}]_{\theta}=[u_{r
i_{r}^{2}}]_{\theta}= \ldots =[u_{r i_{r}^{h}}]_{\theta}$ for $r\in\{1;\cdots;h\}$; as $\{i_{1}^{1};i_{1}^{2}; \ldots
;i_{1}^{h}\}\cap\{i_{2}^{1};i_{2}^{2}; \ldots ;i_{2}^{h}\}\cap, \ldots ,\cap\{i_{h}^{1};i_{h}^{2}; \ldots ;i_{h}^{h}\}\neq \emptyset$
 then, with
 $i\in\{i_{1}^{1};i_{1}^{2}; \ldots ;i_{1}^{h}\}\cap\{i_{2}^{1};i_{2}^{2}; \ldots ;i_{2}^{h}\}\cap, \ldots ,\cap\{i_{h}^{1};i_{h}^{2}; \ldots ;i_{h}^{h}\}$ we have
 $$( m(u_{11},u_{12}, \ldots ,u_{1 h+1}), \ldots ,m(u_{h1},u_{h2}, \ldots ,u_{h h+1}))\in [u_{1
i}]_{\theta}\times\ldots\times[u_{h i}]_{\theta}\subseteq\rho$$ because $\rho$ is $\theta$-closed. Thus $m\in \pol(\rho)$.\\
Finally, if $\rho$ is of type III, there exist $1\leq i_{r}^{1}<i_{r}^{2}< \ldots <i_{r}^{h}\leq h+1, \ 1\leq r \leq h$ such that $[u_{r
i_{r}^{1}}]_{\theta}=[u_{r i_{r}^{2}}]_{\theta}= \ldots =[u_{r i_{r}^{h}}]_{\theta}$ for $r\in\{1;\cdots;h\}$; as $$\{i_{1}^{1};i_{1}^{2}; \ldots
;i_{1}^{h}\}\cap\{i_{2}^{1};i_{2}^{2}; \ldots ;i_{2}^{h}\}\cap, \ldots ,\cap\{i_{h}^{1};i_{h}^{2}; \ldots ;i_{h}^{h}\}\neq \emptyset$$
 then, with
 $i\in\{i_{1}^{1};i_{1}^{2}; \ldots ;i_{1}^{h}\}\cap\{i_{2}^{1};i_{2}^{2}; \ldots ;i_{2}^{h}\}\cap, \ldots ,\cap\{i_{h}^{1};i_{h}^{2}; \ldots ;i_{h}^{h}\}$ we
 have ( with $Z:=(m(u_{11},\ldots,u_{1 h+1}),\ldots,m(u_{h1}, \ldots ,u_{h h+1}))$)
 $$Z\in \{u_{1i},...,u_{hi},T_{[u_{1i}]_{\theta}},...,T_{[u_{hi}]_{\theta}}\}^{h}\subseteq \rho.$$  Hence $m\in \pol(\rho)$.
\end{enumerate}
\end{proof}

\begin{proof}[\textbf{Proof (of proposition \ref{sufficiency-direction})}]\text{ }
Let $f\in \pol(\theta)\setminus(\pol(\theta)\cap \pol(\rho)).$ We
will prove that $G=<\pol(\theta)\cap \pol(\rho)\cup\{f\}>$ is equal
to $\pol(\theta).$ From Lemma \ref{existence-unanimity-all}, it
follows that $\pol(\theta)\cap \pol(\rho)$ contains an
$h$-near-unanimity function $m$. According to Theorem
\ref{baker-pixley} and the fact that  $m\in G$, we have $G=Pol
Inv^{(h)} G$. If $\tau\in Inv^{(h)} G$, then $\pol(\theta)\cap
\pol(\rho)\subseteq <\pol(\theta)\cap \pol(\rho)\cup\{f\}>= G=Pol
Inv^{(h)} G\subseteq \pol(\tau)$. By Lemma
\ref{caracterisation-diagonal-all}, if $\rho$ is of type I or II,
then $\tau$ is either the  empty relation, either a diagonal
relation through $\theta$ or $\rho$; if $\rho$ is of type III and of
order $l$, then $\tau$ is either the empty relation, either a
diagonal relation through $\theta$ or an intersection of relations
of the form $(\rho_{l,\theta})_{\sigma}$ with
$\sigma\in\mathcal{S}_{h}$. Since $f\notin \pol(\rho)$, therefore
$f$ can not preserve an intersection of relations of the form
$(\rho_{l,\theta})_{\sigma}$ with $\sigma\in\mathcal{S}_{h}$;
therefore $\tau$ is the empty relation or a diagonal relation
through $\theta$. In the light of Lemma \ref{caracterisation-I}, we
have $\pol(\theta)\subseteq G$.
\end{proof}

\begin{remark}
Let us mention that if $\rho$ is of type I or II, then $\pol(\theta)\cap \pol(\rho)$ is also maximal in $\pol(\rho)$, whereas if $\rho$ is of type III,
then $\pol(\theta)\cap \pol(\rho)$ is not maximal in $\pol(\rho)$.
\end{remark}

\subsubsection{Proof of the completeness criterion}\label{sec6-subsec5}\text{ }

In this subsection, we show that the relations of type I, II and III are the only  central relations $\rho$ such that $\pol(\theta)\cap \pol(\rho)$  is
maximal in $\pol(\theta)$. So we suppose that $\pol(\theta)\cap \pol(\rho)$  is maximal in $\pol(\theta)$ and we show that except the types announced, in
the other cases we don't have maximality.

\begin{proposition}\label{completeness-criterion}
If $\pol(\rho)\cap \pol(\theta)$ is maximal in $\pol(\theta)$, then $ \rho$ is either of type I, II, or III.
\end{proposition}
Before the proof of Proposition \ref{completeness-criterion}, we will give some results on the properties of the relation $\rho$.
\begin{lemma}\label{major-criterion}
If $\pol(\rho)\cap \pol(\theta)$ is maximal in $\pol(\theta)$, then
$\eta\subseteq \rho$.
\end{lemma}
\begin{proof}By contraposition, suppose that $\eta\nsubseteq \rho$, therefore there exists $$\left(u_1, u_2, \ldots, u_h\right)\in
                                                  \,\eta\quad \text{ such that } \quad \left(u_1, u_2, \ldots, u_h\right)\notin
                                                  \rho.$$ We denote by $\eta^{1}$ the relation $$\eta^{1}=\left\{\left(u_1, u_2, \ldots, u_h\right)\in \rho
/ u_1\theta u_2\right\}.$$ Let's prove  the following inclusions $$\pol(\rho)\cap \pol(\theta)\varsubsetneq \pol(\eta^{1})\varsubsetneq \pol(\theta).$$
\begin{enumerate}
    \item [(i)] Firstly, let us prove  that $\pol(\rho)\cap \pol(\theta)\varsubsetneq \pol(\eta^{1}).$
Let $f\in \pol(\rho)\cap \pol(\theta) $ with arity $n$, we will prove that $f \in \pol(\eta^{1})$.\\
Let $\left(u_{11}, u_{21}, \ldots,u_{h1}\right), \ldots ,\left(u_{1n}, u_{2n}, \ldots,  u_{hn}\right) \in  \,\eta^{1}.$\\ We have $f(u_{11}, \ldots
,u_{1n})\theta f(u_{21}, \ldots ,u_{2n})$ since $f\in \pol(\theta)$ and $$\left(f(u_{11}, \ldots ,u_{1n}),f(u_{21}, \ldots ,u_{2n}), \ldots,  f(u_{h1},
\ldots ,u_{hn})\right)\in \rho$$ since $f\in \pol(\rho)$. Therefore  $f\in \pol(\eta^{1})$.\\  Let $\left(u_1, u_2, \ldots, u_h\right)\in \,\eta\setminus
\rho$ and $(a,b)\notin \theta$ fixed,  we define the $h$-ary operation $f$ on $E_{k}$ by:

\begin{equation}\label{function-for-strict-inclusion-f}
      f(x_1,x_2, \ldots ,x_h)= \left\{\begin{array}{l}
                                                 u_1 \text{ if } (x_1,x_2, \ldots ,x_h)\theta (a,a, \ldots ,a) \\
                                                 u_2 \text{ if } (x_1,x_2, \ldots ,x_h)\theta (a,b,a, \ldots ,a) \\
                                                 \vdots\\
                                                 u_h \text{ if } (x_1,x_2, \ldots ,x_h)\theta (a,a, \ldots ,a,b) \\
                                                 u_1 \text{ otherwise. }
                                               \end{array}
      \right.
\end{equation}
With $\left(u_{11}, u_{21}, \ldots, u_{h1}\right)$, \ldots ,$\left(u_{1s}, u_{2s}, \ldots, u_{ss}\right)$ $\in$ $ \,\eta^{1}$, according to definition of $
\eta^{1}$,  we have $f(u_{11}, \ldots ,u_{1s})= f(u_{21}, \ldots ,u_{2s})$ since $f\in \pol(\theta)$ and $(u_{11}, \ldots ,u_{1n})\theta(u_{21}, \ldots
,u_{2n}).$ From the fact that $\rho$ is totally reflexive we have  $$\left(f(u_{11}, \ldots ,u_{1n}),f(u_{21}, \ldots ,u_{2n}), \ldots, f(u_{h1}, \ldots
,u_{hn})\right)\in \,\eta^{1}.$$  Hence $f\in \pol(\eta^{1})$.

Furthermore, $\left(\begin{array}{cccccc} a&a&a&\cdots &a&a \\a&b&a&\cdots&a&a \\\vdots \\a&a&a&\cdots&a&b \\
\end{array}\right)\subseteq \rho$ and by the definition of $f$ we have
$$\left(u_1, u_2, \ldots, u_h\right)=\left(f(a,a,a,\ldots,a), f(a,b,a,\ldots,a), \ldots, f(a,a,\ldots,a,b)\right)\notin \rho.$$  Therefore, $f\notin \pol(\rho) $ and $\pol(\rho)\cap \pol(\theta)\varsubsetneq \pol(\eta^{1}).$
    \item [(ii)]Secondly we will show that $\pol(\eta^{1})\varsubsetneq \pol(\theta)$.\\ From the equality $pr_{12}(\eta^{1})=\theta$, it
    follows that $\pol(\eta^{1})\subseteq \pol(\theta)$.  Let $\left(u_1, u_2, \ldots, u_h\right)\in \,\eta\setminus \rho$ be a
     fixed element and $c$ be a central element of $\rho.$ It is obvious that $\left(u_1, u_2, \ldots, u_h\right)\notin \,\eta^{1}$.
Let $a,b,c\in E_{k}$ such that $(a,b)\in\theta$, $a\neq b$ and $(a,c)\notin\theta$.\\
For $1\leq i\leq h$, we define the $(h-1)$-tuple $\mathbf{W}_{i}=(w_{i1},\ldots,w_{i h-1})$ by:
\[w_{1l}=a\text{ for } 1\leq l\leq h-1\]
\[w_{2l}=\left\{\begin{array}{cc}
                  b & \text{ if } l=h-1 \\
                  a & \text{ elsewhere }
                \end{array}
\right.\]
\[\text{for } m\geq 3, w_{ml}=\left\{\begin{array}{cc}
                  c & \text{ if } l=m-2 \\
                  a & \text{ elsewhere. }
                \end{array}
\right.\] We have $\mathbf{W}_{1}\neq \mathbf{W}_{2}$ and for every
$1\leq i<j\leq h$, $\mathbf{W}_{i}\theta \mathbf{W}_{j}$ if and only
if $i=1$ and $j=2$. We define the $(h-1)$-ary operation $g$ on
$E_{k}$ by:
\[g(x_1,x_2, \ldots ,x_{h-1})=  \left\{\begin{array}{l}
                                                 u_1 \text{ if } (x_1,x_2, \ldots ,x_{h-1})=\mathbf{W}_{1}  \\
                                                 u_2 \text{ if }(x_1,x_2, \ldots ,x_{h-1})=\mathbf{W}_{2}\\
                                                 u_i \text{ if } (x_1,x_2, \ldots ,x_{h-1})\theta \mathbf{W}_{i}\text{ for some } 3\leq i\leq h-1  \\
                                                 u_{h} \text{ elsewhere. }
                                               \end{array}
      \right.\]
Since  $g(\theta)\subseteq \{(u_1,u_2);(u_2,u_1)\}\cup\{(u_i,u_i): i\in\{1;\cdots;h\}\}$, $g\in \pol(\theta)$.\\
The matrix $(w_{ij})_{\begin{subarray}{l}
  1\leq i\leq h\\
  1\leq j\leq h-1
\end{subarray}
}$ is a subset of $\eta^{1}$ and \[g\left((w_{ij})_{\begin{subarray}{l}
  1\leq i\leq h\\
  1\leq j\leq h-1
\end{subarray}
}\right)=g((\mathbf{W}_{1}),\ldots,g(\mathbf{W}_{h}))=(u_1,\ldots,u_h)\notin\eta^{1}.\] Hence $g\notin \pol(\eta^{1})$. Therefore
$\pol(\eta^{1})\varsubsetneq \pol(\theta)$.
\end{enumerate}
\end{proof}

\begin{remark}\label{arity-criterion}
If $\pol(\theta)\cap \pol(\rho)$ is maximal in $\pol(\theta)$,  then the arity of $\rho$ is less than or equal to $t$ where $t$ is the  number of
equivalence classes of $\theta$. Indeed, if the arity of $\rho$ is greater than $t$ then $\eta\nsubseteq \rho$, because $\rho\neq
E_k^{\textit{arity}(\rho)}$ and $(\eta\subseteq\rho\Rightarrow\rho=E_k^{\textit{arity}(\rho)})$.
\end{remark}
From now on we suppose that $\eta\subseteq\rho$.

According to the definition of $\rho_{0,\theta}$ we have $\rho\subseteq\rho_{0,\theta}$.

\begin{lemma}\label{lem-rec-O}
If $\rho= \rho_{0,\theta}$, then $\pol(\theta)\cap \pol(\rho)$ is  maximal in $\pol(\theta)$
\end{lemma}
\begin{proof}
If $\rho=\, \rho_{0,\theta}$, then $\rho$ is $\theta$-closed. Hence $\rho$ is of type II and $\pol(\theta)\cap \pol(\rho)$ is  maximal in $\pol(\theta)$
from Proposition \ref{sufficiency-direction}.
\end{proof}
Now we suppose that $\rho\varsubsetneq\, \rho_{0,\theta}$ and we
have two cases express in the following lemmas.
\begin{lemma}\label{lem-rec-I}
If $\rho\varsubsetneq\, \rho_{0,\theta}\varsubsetneq E_{k}^{h}$, then $\pol(\theta)\cap \pol(\rho)$ is not maximal in $\pol(\theta)$
\end{lemma}
\begin{proof}It is easy to see that $\pol(\rho)\cap \pol(\theta)\subseteq \pol(\theta)\cap \pol(\rho_{0,\theta})\subseteq \pol(\theta).$\\
Let $a,b\in  E_{k}$ such that $(a,b)\notin \theta$ and  $\left( u_1, u_2, \ldots, u_h\right)\in\, \rho_{0,\theta}\setminus \rho$ (respectively
$E_{k}^{h}\setminus\rho$). Using the $h$-ary operation $f$ defined in the proof of Lemma \ref{major-criterion}, we show that $\pol(\rho)\cap
\pol(\theta)\varsubsetneq
\pol(\theta)\cap \pol(\rho_{0,\theta})$(respectively $\pol(\theta)\cap \pol(\rho_{0,\theta})\varsubsetneq \pol(\theta).$)\\
Hence $\pol(\rho)\cap \pol(\theta)\varsubsetneq \pol(\theta)\cap \pol(\rho_{0,\theta})\varsubsetneq \pol(\theta).$
\end{proof}

\begin{lemma}\label{lem-rec-II}\text{ }
If $\rho\varsubsetneq\,\rho_{0,\theta}= E_{k}^{h}$  and there exists an integer $l>h$ such that  $\rho_{0,\theta}^{l}\neq E_k^{l}$, then $\pol(\theta)\cap
\pol(\rho)$ is not maximal in $\pol(\theta)$ where $\rho_{0,\theta}^{l}$ is the relation  $$\rho_{0,\theta}^{l}=\left\{\mathbf{u}\in E_{k}^{l} / \exists
(e_1,\ldots,e_l)\in [u_1]_\theta\times\cdots\times[u_l]_\theta; \{e_1,\ldots,e_l\}^{h}\subseteq \rho\right\}.$$
\end{lemma}
\begin{proof} To prove our lemma, we will show that $\pol(\rho)\cap \pol(\theta)\varsubsetneq
\pol(\rho_{0,\theta}^{m})\cap \pol(\theta)\varsubsetneq \pol(\theta)$ where $m=min\{l\in \mathbb{N}\setminus\{0;1;2; \ldots ;h\}; \,\rho_{0,\theta}^{l}\neq
E_k^{l}\}$. We will distinguish two cases;  one for each inclusion.
\begin{enumerate}
    \item [(i)]Let us prove here that $\pol(\rho)\cap \pol(\theta)\varsubsetneq
\pol(\rho_{0,\theta}^{m})\cap \pol(\theta)$ \\
Let $f\in \mathcal{O}^{n}(E_k)$ be an $n$-ary operation on $E_k$ such that $f\in \pol(\theta)\cap \pol(\rho)$. According to the definition of
$\rho_{0,\theta}^{m}$, for
$$\left(u_{11},  \ldots, u_{m1}\right), \ldots ,\left(u_{1n}, \ldots, u_{mn} \right)\in\, \rho_{0,\theta}^{m}$$ there exist $a_{ij}\in [u_{ij}]_{\theta}$ such that for all
$j\in\{1;\cdots;n\}$, we have\\
$\{a_{1j},\ldots,a_{mj}\}^{h}\subseteq\rho$. Then, it follows that
there exist $a_{ij}\in [u_{ij}]_{\theta}$ such that for all
$i_1,i_2, \ldots ,i_h\in\{1;\cdots;m\}$ we have,
$$\left(a_{i_{1}1}, a_{i_{2}1}, \ldots,  a_{i_{h}1} \right), \ldots ,\left(a_{i_{1}n},  a_{i_{2}n}, \ldots, a_{i_{h}n}\right) \in \rho;$$
which imply that there exist $a_{ij}\in [u_{ij}]_{\theta}$ such that for all  $i_1,i_2, \ldots ,i_h\in\{1;\cdots;m\}$ we have,
$$\left(f(a_{i_{1}1}, \ldots ,a_{i_{1}n}), f(a_{i_{2}1}, \ldots ,a_{i_{2}n}), \ldots, f(a_{i_{h}1}, \ldots ,a_{i_{h}n})\right)\in \rho.$$
Since $f\in \pol(\theta)$ and for all $ i \in \{1;\cdots;m\}$, $(u_{i1}, \ldots ,u_{in})\theta(a_{i1}, \ldots ,a_{in})$, it appears that
$$\left(f(u_{11}, \ldots ,u_{1n}),  \ldots , f(u_{m1}, \ldots ,u_{mn})\right)\in\, \rho_{0,\theta}^{m} \quad \text{and} \quad f\in\, \rho_{0,\theta}^{m}.$$ Hence $f\in \pol(\rho_{0,\theta}^{m})\cap \pol(\theta).$

Let $\left(v_1,  \ldots,  v_h\right)\in E_{k}^{h}\setminus \rho$ and $(a,b)\notin \theta$, using the $h$-ary operation $f$ on $E_k$, specified by:
$$f(x_1,x_2, \ldots ,x_h)=  \left\{\begin{array}{l}
                                                 v_1 \text{ if } (x_1,x_2, \ldots ,x_h)\theta (a,a, \ldots ,a) \\
                                                 v_2 \text{ if } (x_1,x_2, \ldots ,x_h)\theta (a,b,a, \ldots ,a) \\
                                                 \vdots\\
                                                 v_h \text{ if } (x_1,x_2, \ldots ,x_h)\theta (a,a, \ldots ,a,b,a) \\
                                                 v_h \text{ if } (x_1,x_2, \ldots ,x_h)\theta (a,a, \ldots ,a,a,b) \\
                                                 v_1 \text{ otherwise. }
                                               \end{array}
      \right.$$
and the fact that $m>h$ and $\rho_{0,\theta}^{m}$ is totally reflexive,  we obtain $f\notin \pol(\rho)$ and $f\in \pol(\rho_{0,\theta}^{m}).$

    \item [(ii)] This item is devoted to prove that $\pol(\rho_{0,\theta}^{m})\cap \pol(\theta)\varsubsetneq \pol(\theta).$\\
Let $\left(u_{1}, \ldots, u_{m}\right)\in E_k^{m}\setminus\, \rho_{0,\theta}^{m}$ ($E_k^{m}\setminus\, \rho_{0,\theta}^{m}$ is not empty because
$\rho_{0,\theta}^{m}\neq E_{k}^{m}$) and $(a,b)\notin \theta.$ Let us consider the $m$-ary operation $h$ on $E_{k}$ defined by:
\begin{equation}\label{function-for-strict-inclusion-h-in-theta}
    h(x_1, \ldots ,x_m)=
\left\{\begin{array}{l}
                                                 u_1 \text{ if } (x_1, \ldots ,x_m)\theta (a, \ldots ,a,a)=a_1 \\
                                                 u_2 \text{ if } (x_1, \ldots ,x_m)\theta (a,b,a, \ldots ,,a)=a_2 \\
                                                 u_3 \text{ if } (x_1, \ldots ,x_m)\theta (a,a,b,a \ldots ,a)=a_3 \\
                                                 \vdots\\
                                                 u_{m} \text{ if } (x_1, \ldots ,x_m)\theta (a,a,a, \ldots a,b)=a_m\\
                                                 u_m \text{ otherwise. }
                                               \end{array}
      \right.
\end{equation}
The implication $a_i\theta a_j\Rightarrow i=j$  allows us to say that $h\in \pol(\theta).$ It is clear that
$$\left(\begin{array}{c} a \\\vdots\\a\\a\\a \\a\\a \\
\end{array}\right), \  \left(\begin{array}{c} a \\b\\a \\\vdots\\a\\a\\a \\
\end{array}\right), \ \left(\begin{array}{c} a \\a\\b\\a\\\vdots \\a\\a \\
\end{array}\right), \  \ldots , \ \left(\begin{array}{c} a \\a\\a\\\vdots\\a\\a \\b \\
\end{array}\right)\in\, \rho_{0,\theta}^{m}$$ but following the definition  of $h$, it appears that $$\left(u_1, u_2, u_3,\ldots, u_m\right)=
\left( h(a_1), h(a_2), h(a_3), \ldots, h(a_m)\right)\notin\, \rho_{0,\theta}^{m};$$ Therefore, $h\notin \pol(\rho_{0,\theta}^{m})$ and
$\pol(\rho_{0,\theta}^{m})\cap \pol(\theta)\varsubsetneq \pol(\theta).$
\end{enumerate}
Hence $\pol(\rho)\cap \pol(\theta)\varsubsetneq \pol(\rho_{0,\theta}^{m})\cap \pol(\theta)\varsubsetneq \pol(\theta).$
\end{proof}

In the light of Lemma \ref{lem-rec-II}, it is natural to suppose
$\rho_{0,\theta}= E_{k}^{h}$ and $\rho_{0,\theta}^{l}= E_k^{l}$ for
all $l>h.$ In particular, for $l=t$ we have $\rho_{0,\theta}^{t}=
E_k^{t}$ and there exists $(x_1,x_2, \ldots ,x_t)\in C_0\times
C_1\times \ldots \times C_{t-1}$ such that $\{x_1,x_2, \ldots
,x_t\}^{h} \subseteq\rho$. Let $\,\varsigma$ be the relation
$$\, \varsigma=\left\{\mathbf{a}\in E_{k}^{h}| \exists u_{ij}\in [a_{i}]_{\theta}, 1\leq i,j\leq h\text{ with } i\neq j\right.\text{ such that } \forall j\in\{1;\cdots;h\},$$
$$\left. \left(a_j, u_{j_{1}j}, u_{j_{2}j}, \ldots, u_{j_{h-1}j}\right)\in \rho, \text{ with } \{j_{1}; \cdots; j_{h-1}\}=\{1; \cdots; h\}\setminus\{j\}\right\}.$$
Since $\rho $ is totally symmetric, $\, \varsigma=\underset{\sigma\in\mathcal{S}_{h}}\bigcap(\rho_{1,\theta})_{\sigma}$ and it is obvious that
$\rho\subseteq\, \varsigma$.

\begin{lemma}\label{lem-rec-I-gamma3-eal}
If $\rho= \varsigma$, then $\pol(\theta)\cap \pol(\rho)$ is maximal
in $\pol(\theta)$.
\end{lemma}
\begin{proof}
If $\rho= \varsigma$, then $\rho$ is weakly $\theta$-closed of order $1$. Hence $\rho$ is a relation of type III. By Proposition
\ref{sufficiency-direction} $\pol(\theta)\cap \pol(\rho)$ is maximal in $\pol(\theta)$.
\end{proof}
\begin{lemma}\label{lem-rec-I-gamma3}
If $\rho\varsubsetneq\, \varsigma\varsubsetneq E_{k}^{h}$ then $\pol(\theta)\cap \pol(\rho)$ is not maximal in $\pol(\theta)$
\end{lemma}

\begin{proof}
We will prove that $\pol(\rho)\cap \pol(\theta)\varsubsetneq \pol(\theta)\cap \pol(\, \varsigma)\varsubsetneq \pol(\theta).$ Let  $f\in \pol(\theta)\cap
\pol(\rho)$, $n$-ary. \\ Let $\left(a_{11}, a_{21},\ldots, a_{h1}\right), \ldots ,\left(a_{1n}, a_{2n}, \ldots, a_{sn}\right)\in \, \varsigma$, we will
show that
$$\left(f(a_{11}, \ldots ,a_{1n}), f(a_{21}, \ldots ,a_{2n}), \ldots, f(a_{h1}, \ldots ,a_{hn}) \right)\in \, \varsigma.$$

Firstly, we show that $\left(f(a_{11}, \ldots ,a_{1n}), \ldots, f(a_{h1}, \ldots ,a_{hn}) \right)\in \rho_{1,\theta}$. Since $\left(a_{1i},\ldots,
a_{hi}\right)\in \, \varsigma$ for $1\leq i\leq n$, there exist $u_{ji}\in[a_{ji}]_{\theta}$, $2\leq j\leq h$ such that $\left(a_{1i}, u_{2i},\ldots,
u_{hi}\right)\in\rho$ for $1\leq i\leq n$. Therefore $$\left(f(a_{11}, \ldots ,a_{1n}), f(u_{21}, \ldots ,u_{2n}), \ldots, f(u_{h1}, \ldots ,u_{hn})
\right)\in\rho$$ and $f(u_{j1}, \ldots ,u_{jn})\theta f(a_{j1}, \ldots ,a_{jn})$ for $2\leq j \leq n$. Hence $$\left(f(a_{11}, \ldots ,a_{1n}), f(a_{21},
\ldots ,a_{2n}), \ldots, f(a_{h1}, \ldots ,a_{hn}) \right)\in \rho_{1,\theta}.$$ Secondly, we show that $\left(f(a_{11}, \ldots ,a_{1n}), \ldots, f(a_{h1},
\ldots ,a_{hn}) \right)\in (\rho_{1,\theta})_{\sigma}$ for $\sigma\in \mathcal{S}_{h}$. Let $\sigma\in \mathcal{S}_{h}$. Since $\left(a_{1i},\ldots,
a_{hi}\right)\in \, (\rho_{1,\theta})_{\sigma}$ for $1\leq i\leq n$, it follows that $\left(a_{\sigma^{-1}(1)i},\ldots, a_{\sigma^{-1}(h)i}\right)\in \,
\rho_{1,\theta}$ for $1\leq i\leq n$. Hence $$\left(f(a_{\sigma^{-1}(1)1}, \ldots ,a_{\sigma^{-1}(1)n}), \ldots, f(a_{\sigma^{-1}(h)1}, \ldots
,a_{\sigma^{-1}(h)n}) \right)\in \rho_{1,\theta}.$$ Therefore  $\left(f(a_{11}, \ldots ,a_{1n}), \ldots, f(a_{h1}, \ldots ,a_{hn}) \right)\in
(\rho_{1,\theta})_{\sigma}.$\\
Thus $\left(f(a_{11}, \ldots ,a_{1n}), \ldots, f(a_{h1}, \ldots ,a_{hn}) \right)\in \, \varsigma.$\\
And it follows that  $\pol(\rho)\cap \pol(\theta)\subseteq \pol(\theta)\cap \pol(\, \varsigma).$ \\
As  $\rho\varsubsetneq\, \varsigma$ we have $ \varsigma\setminus
\rho\neq \emptyset;$ then with $\left( u_1, u_2, \ldots,
u_h\right)\in\, \varsigma\setminus \rho$ and $(a,b)\notin \theta$
fixed; using the $h$-ary operation $f$ defined by
(\ref{function-for-strict-inclusion-f}), and the same argument used
in the proof of Lemma \ref{lem-rec-I}, it
follows that $f\in \pol(\, \varsigma)$ and $f\notin \pol(\rho).$ Then $\pol(\rho)\cap \pol(\theta)\varsubsetneq \pol(\theta)\cap \pol(\, \varsigma).$ \\
It is obvious to see that $\pol(\theta)\cap \pol(\, \varsigma)\varsubsetneq
\pol(\theta)$ (because $\, \varsigma\neq E_{k}^{h}$). \\
Let us take a fixed element $\left(u_1, u_2, \ldots, u_h\right)\in
E_{k}^{h}\setminus\, \varsigma$ and  $(a,b)\notin \theta$. Using the
operation $f$ define above, we easily obtain $f\in \pol(\theta)$.\\
On the other hand
$$\left(\begin{array}{cccccc} a&a&a&\cdots&a&a \\a&b&a&\cdots&a&a \\&&\cdots&\cdots&& \\a&a&\cdots&a&b&a \\a&a&\cdots&a&a&b \\
\end{array}\right)\subseteq \, \varsigma$$ but by the definition of $f$ we have
$$\left(u_1, u_2, \ldots, u_h\right)=\left(f(a,a, \ldots ,a), f(a,b,a, \ldots ,a), \ldots, f(a,a, \ldots ,a,a,b)\right)\notin \, \varsigma.$$ It follows that $f\notin \pol(\, \varsigma);$
thus $\pol(\rho)\cap \pol(\theta)\varsubsetneq \pol(\, \varsigma)\cap \pol(\theta)\varsubsetneq \pol(\theta)$.

\end{proof}

\begin{lemma}\label{lem-rec-II-gamma3}\text{ }
If $\rho\varsubsetneq\, \varsigma= E_{k}^{h}$  and there exists
$h<l\leq t$ with $\, \varsigma^{l}\neq E_k^{l}$, then
$\pol(\theta)\cap \pol(\rho)$ is not maximal in $\pol(\theta)$ where
$\,
\varsigma^{l}=\underset{\sigma\in\mathcal{S}_{l}}\bigcap(\rho_{1,\theta}^{l})_{\sigma}$
with  $$\, \rho_{1,\theta}^{l}=\left\{\left(a_1, \ldots,  a_l
\right)\in E_{k}^{l} / \exists u_{i}\in [a_{i}]_{\theta}, 2\leq
i\leq l: \{a_1;u_{2},u_{3},\ldots,u_{l}\}^{h}\subseteq
 \rho\right\}.$$
\end{lemma}

\begin{proof}
We set $m:=min\{l\in \{0;1;\cdots;t\}\setminus\{0;1;2; \ldots
;h\}/\, \varsigma^{l}\neq E_k^{l}\}$. We will  prove that
$\pol(\rho)\cap \pol(\theta)\varsubsetneq \pol(\, \varsigma^{m})\cap
\pol(\theta)\varsubsetneq \pol(\theta)$.\\ We use the similar
argument as in the proof of Lemma \ref{lem-rec-I-gamma3} to show
that $\pol(\rho)\cap \pol(\theta)\subseteq \pol(\, \varsigma^{m})\cap \pol(\theta)$.\\
Let $\left( u_1, u_2, \ldots, u_h\right)\in\, \varsigma\setminus \rho$ and $(a,b)\notin \theta$ fixed.

Using the $h$-ary operation $f$ defined by (\ref{function-for-strict-inclusion-f}), we have again $f\notin  \pol(\rho)$ and $f\in \pol(\, \varsigma^{m})$;
since $\left( u_1, u_2, \ldots, u_h\right)\in E_{k}^{h}=\, \varsigma$, there exist $ a_{ij}\in [u_{i}]_{\theta}, i,j\in\{1,..,s\}\text{ with } i\neq j$
such that for all $j\in\{1, \ldots ,h\},$
$$ \left(u_j,  a_{j_{1}j}, a_{j_{2}j}, \ldots, a_{j_{h-1}j}\right)\in \rho$$  with $\{j_{1}, \ldots ,j_{h-1}\}=\{1, \ldots ,h\}\setminus\{j\}$.
Then it follows that $\{u_1,u_2, \ldots ,u_h\}^{m}\subseteq \, \varsigma^m$. \\
Moreover, as we have $\varsigma^{m}\neq E_{k}^{m}$, there exists $\left(u_{1}, \ldots, u_{m}\right)\in E_k^{m}\setminus \, \varsigma^{m} $. Given
$\left(u_{1}, \ldots, u_{m}\right)\in E_k^{m}\setminus \, \varsigma^{m}$ and  $(a,b)\notin \theta$.

Using the $m$-ary operation $h$ defined by (\ref{function-for-strict-inclusion-h-in-theta}), and the same argument used in the proof of Lemma
\ref{lem-rec-II}, we have $h\in \pol(\theta)$ but $h\notin \pol(\,\varsigma^{m})$. Thus $\pol(\rho)\cap \pol(\theta)\varsubsetneq \pol(\,
\varsigma^{m})\cap \pol(\theta)\varsubsetneq \pol(\theta)$.
\end{proof}

The previous lemma suggests us to suppose that for each
$l\in\{h;\cdots; t\}$, $\varsigma^{l}= E_k^l$.

Let $m=max\{|C_i|; 0\leq i\leq t-1\}$ and we denote by $\, \gamma'$
the relation
$$\, \gamma'=\left\{\left(a_1, \ldots, a_m, a_{m+1}, \ldots,  a_{m+t-1} \right)\in E_{k}^{m+t-1} / \forall \{i;j\}\subseteq \{1;\cdots;m\}, a_i\theta a_j;
\right.$$
$$\left. \qquad \exists u_{i}\in [a_{i}]_{\theta} ,m+1\leq i \leq m+t-1, \forall 1\leq i\leq m,\right.$$
$$\left.\{a_i;u_{m+1}; \ldots ;u_{m+t-1}\}^{h}\subseteq \rho\right\}.$$
We have $\pol(\, \gamma')\subseteq \pol(\theta)$. We define two relations $\epsilon_{=}$ and $\epsilon'_{\theta}$ on $\{1;\cdots;m+t-1\}$ by:
$$(i,j)\in\epsilon_{=}\Leftrightarrow i=j \text { and } (i,j)\in \epsilon'_{\theta}\Leftrightarrow (i=j \text{ or } \{i,j\}\subseteq \{1; \ldots ;m\}).$$
$\epsilon_{=}$ and $\epsilon'_{\theta}$ are  equivalence relations and $\, \gamma' \subseteq D_{\epsilon_{=}\epsilon'_{\theta}}$.

\begin{lemma}\label{lem-rec-II-gamma3'}\text{ }
If $\forall l\in\{h,h+1,...,t\}$, $\, \varsigma^{l}= E_{k}^{h}$ and $\, \gamma'\neq D_{\epsilon_{=}\epsilon'_{\theta}}$, then $\pol(\theta)\cap \pol(\rho)$
is not maximal in $\pol(\theta)$.
\end{lemma}
\begin{proof} We just have to prove that $\pol(\rho)\cap \pol(\theta)\varsubsetneq \pol(\, \gamma') \varsubsetneq \pol(\theta)$. \\
Let $f\in \mathcal{O}^{n}(E_k)$ be an $n$-ary operation on $E_k$ such that $f\in \pol(\theta)\cap \pol(\rho)$. Let $$\left(a_{11}, \ldots, a_{m+t-1 1}
\right), \ldots ,\left(a_{1n}, \ldots, a_{m+t-1 n} \right)\in \, \gamma'.$$ According to the  definition of $\, \gamma'$, for all $1\leq j \leq n$, for all
$\alpha,\beta\in \{1;\cdots;m\}$, we have $[a_{\alpha j}]_{\theta}=[a_{\beta j}]_{\theta}$ and for all $j\in \{1;\cdots;n\}$, there exists $u_{rj}\in
[a_{rj}]_{\theta}$ with $r\in \{m+1;\cdots;m+t-1\}$ such that for all  $i\in \{1;\cdots;m\}$,
$$\{a_{ij};u_{m+1j};u_{m+2j};\ldots;u_{m+t-1j}\}^{h}\subseteq
 \rho.$$  Since $f\in \pol(\theta)\cap
\pol(\rho)$, for all $\alpha,\beta\in \{1;\cdots;m\}$, $[f(a_{\alpha 1}, \ldots ,a_{\alpha n})]_{\theta}=[f(a_{\beta 1}, \ldots ,a_{\beta n})]_{\theta}$
and for all  $r\in \{m+1;\cdots;m+t-1\}$, $f(u_{r1}, \ldots ,u_{rn})\in [f(a_{r1}, \ldots ,a_{rn})]_{\theta}.$ It follows that for all $i\in\{1;\cdots;m\}$
 $$\{f(a_{i1}, \ldots ,a_{in});f(u_{m+11}, \ldots ,u_{m+1n}); \ldots ;f(u_{m+t-11}, \ldots ,u_{m+t-1n})\}^{h}\subseteq\rho
 .$$
 Consequently $$\left(f(a_{11}, \ldots ,a_{1n}), \ldots, f(a_{m+t-11}, \ldots ,a_{m+t-1n})\right)\in \, \gamma'$$ and  then
$f\in \pol(\, \gamma').$


Let $\left(u_1, u_2, \ldots, u_h\right)\in\, E_{k}^{h}\setminus \rho$ and $(a,b)\notin \theta$ fixed. Using the $h$-ary operation $f$ defined by
(\ref{function-for-strict-inclusion-f}), we have $f\notin  \pol(\rho)$ and $f\in \pol(\, \gamma')$; since $\left(u_1, u_2, \ldots, u_h\right)\in
E_{k}^{h}=\,
\varsigma$, it follows that there exist $ a_{ij}\in [u_{i}]_{\theta}; i,j\in\{1;\cdots; h\}\text{ with } i\neq j$  such that for all $j\in\{1;\cdots; h\},$%
$$ \left(u_j, a_{j_{1}j}, a_{j_{2}j}, \ldots, a_{j_{h-1}j} \right)\in \rho,$$ \text{ with } $\{j_{1}, \ldots ,j_{h-1}\}=\{1;\cdots; h\}\setminus\{j\}$.
It follows that $$\bigcup\limits_{i\in\{1;\cdots;h\}} \{u_i\}^{m}\times\{u_{l}; l\in \{1;\cdots;h\}\setminus\{i\}\}^{t-1}\subseteq\, \gamma'.$$ From the
fact that $\, \gamma'\neq D_{\epsilon_{=}\epsilon'_{\theta}}$, we can deduce that $ D_{\epsilon_{=}\epsilon'_{\theta}}\setminus \, \gamma'\neq \emptyset
$.\\ Given  $\left(v_{1}, \ldots, v_{m+t-1}\right)\in D_{\epsilon_{=}\epsilon'_{\theta}}\setminus \, \gamma'$, $(a,b)\notin \theta$ and consider the
following $(m+t-1)$-ary operation defined on $E_k$ by: $$ h(x_1, \ldots ,x_{m+t-1})=  \left\{\begin{array}{ll}
                                                 v_1& \text{ if } (x_1, \ldots ,x_{m+t-1})\theta (a, \ldots ,a,a)=a_1 \\
                                                 v_{m+1}& \text{ if } (x_1, \ldots ,x_{m+t-1})\theta (a,b,a, \ldots ,,a)=a_2 \\
                                                 v_{m+2} &\text{ if } (x_1, \ldots ,x_{m+t-1})\theta (a,a,b,a \ldots ,a)=a_3 \\
                                                 &\vdots\\
                                                 v_{m+t-1} &\text{ if } (x_1, \ldots ,x_{m+t-1})\theta (a,a,a, \ldots a,b)=a_t\\
                                                 v_{m+t-1} &\text{ otherwise. }
                                               \end{array}
      \right.
$$
Since $a_i\theta a_j\Rightarrow i=j$, $h\in \pol(\theta).$ From the fact that $\{a,b\}^{h}\subseteq\rho$ we have $h\notin
\pol(\gamma').$ Indeed, \\
$$\left(\begin{array}{c} a \\\vdots\\a \\\vdots\\a\\a\\a \\a\\a \\
\end{array}\right), \left(\begin{array}{c} a \\\vdots\\a \\b\\a \\\vdots\\a\\a\\a \\
\end{array}\right), \left(\begin{array}{c} a \\\vdots\\a \\a\\b\\a\\\vdots \\a\\a \\
\end{array}\right), \ldots ,\left(\begin{array}{c}a \\\vdots\\ a \\a\\a\\\vdots\\a\\a \\b \\
\end{array}\right) \in \, \gamma'  $$
 but
$$\left(v_1,\ldots,v_1, v_{m+1},\ldots, v_{m+t-1}\right)=\left(h(a_1), \ldots, h(a_1), h(a_2), \ldots, h(a_{t})\right)\notin \, \gamma'.$$
\end{proof}
In our next step, we assume that
\begin{equation}\label{transversale-classe}
\, \gamma'= D_{\epsilon_{=}\epsilon'_{\theta}}\end{equation} i.e.,  for all $i\in\{0;\cdots; t-1\},$ there exists $u_{ji}\in C_j, j\in\{0;\cdots;
t-1\}\setminus \{i\}$ such that for all $a\in C_{i},$ $\{a;u_{1i}; \ldots ;u_{i-1 i},u_{i+1 i}; \ldots ;u_{ti}\}^{h}\subseteq\rho.$ Let $\,\zeta$ be the
relation defined by:
$$\,\zeta=\left\{\mathbf{a}\in E_{k}^{h}| \exists u_i\in[a_i]_{\theta},\right. \left.\text{ such that } \left(a_i, u_{i_1},
\ldots, u_{i_{h-1}}\right)\in \rho, i\in\{1;\cdots;h\} \right. $$
$$\left.\text{ and } \{i_1; \ldots
;i_{h-1}\}=\{1;\cdots;h\}\setminus\{i\}\right\}.$$ Our goal is to
show that  if $\pol(\theta)\cap \pol(\rho)$ is maximal in
$\pol(\theta)$, then $\,\zeta=E_{k}^{h}$. Before that, let us prove
the following result.
\begin{lemma}\label{lem19}
If $\rho=\,\zeta$, then $\rho=E_{k}^{h}$.
\end{lemma}
\begin{proof} We suppose that $\rho=\,\zeta$. Let $a=(a_1,a_2, \ldots ,a_s)\in E_{k}^{h}.$
If there exists $\{i;j\}\subseteq\{1;\cdots;h\}$ with  $i\neq j$ such that $(a_i,a_j)\in\theta$ then $a\in \rho$ since $\eta\subseteq \rho.$ Else,
according to (\ref{transversale-classe}), for all $i\in\{1,...,t\}$, there exist $u_{ji}\in C_j, j\in\{0;\cdots; t-1\}\setminus \{i\}$ such that for all
$b\in C_{i}$, $$\{b;u_{0i}; \ldots ;u_{i-1i},u_{i+1i}; \ldots ;u_{t-1 i}\}^{h}\subseteq\rho $$ where  $C_j=[a_j]_{\theta}, j\in\{1;\cdots;h\}$. Hence
$$(C_i\cup\{u_{0i},...,u_{t-1 i}\}\setminus\{u_{ii}\})^{h}\subseteq\rho, 0\leq i\leq t-1.$$ Therefore
$$\left(a_1,  a_2, \ldots, a_h\right)\in \,\zeta=\rho.$$ It follows that $\rho= E_{k}^{h}$.
\end{proof}

From Lemma \ref{lem19} and the fact that $\rho\neq E_{k}^{h}$,   we obtain $\rho\varsubsetneq \,\zeta\subseteq E_{k}^{h}$. Now we can prove the following
lemma.
\begin{lemma}\label{lem-rec-O-gamma4} If $\,\zeta\neq E_{k}^{h}$, then $\pol(\theta)\cap \pol(\rho)$ is not maximal in $\pol(\theta)$.
\end{lemma}
\begin{proof}We have to prove that $\pol(\rho)\cap \pol(\theta)\varsubsetneq \pol(\theta)\cap \pol(\,\zeta)\varsubsetneq \pol(\theta).$
Let $f\in \mathcal{O}^{n}(E_k)$ be an $n$-ary operation on $E_k$ such that $f\in \pol(\theta)\cap \pol(\rho)$. Let $$\left(a_{11}, a_{21}, \ldots,
a_{h1}\right), \ldots ,\left(a_{1n}, a_{2n}, \ldots, a_{hn} \right)\in \,\zeta.$$ From the definition of $\,\zeta$,  for all $i\in\{1;\cdots;h\}$, for all
$j\in\{1, \ldots ,n\}$, there exist $u_{ij}\in [a_{ij}]_{\theta}$ such that
$$\left(a_{1j}, u_{2j}, \ldots, u_{hj}\right),\left(u_{1j}, a_{2j}, \ldots, u_{hj}\right), \ldots ,\left(u_{1j}, u_{2j}, \ldots, u_{(h-1)j}, a_{hj} \right)\in\rho;$$ and as $f\in
\pol(\rho)$ we have $$\left(f(a_{11}, \ldots ,a_{1n}), f(u_{21}, \ldots ,u_{2n}), \ldots,  f(u_{h1}, \ldots ,u_{hn}) \right)\in \rho,$$ $$\left(f(u_{11},
\ldots ,u_{1n}), f(a_{21}, \ldots ,a_{2n}), f(u_{31}, \ldots ,u_{3n}), \ldots,  f(u_{h1}, \ldots ,u_{hn})\right)\in \rho, \ldots,$$ $$\left(f(u_{11},
\ldots ,u_{1n}),  \ldots, f(u_{(s-1)1}, \ldots ,u_{(h-1)n}), f(a_{h1}, \ldots ,a_{hn}) \right)\in \rho.$$ Moreover,  as $f\in \pol(\theta)$, we have
$f(u_{i1}, \ldots ,u_{in})\theta f(a_{i1}, \ldots ,a_{in})$; therefore,  using the definition of $\,\zeta$,
$$\left(f(a_{11}, \ldots ,a_{1n}), f(a_{21}, \ldots ,a_{2n}), \ldots, f(a_{h1}, \ldots ,a_{hn})\right)\in \,\zeta.$$ Then $\pol(\rho)\cap \pol(\theta)\subseteq
\pol(\theta)\cap \pol(\,\zeta)$. \\
As  $\rho\varsubsetneq \,\zeta$ we have  $ \,\zeta\setminus \rho$ is not empty. Given $\left( u_1, u_2, \ldots, u_h\right)\in \,\zeta\setminus \rho$ and
$(a,b)\notin \theta$, using the $h$-ary operation $f$ defined by (\ref{function-for-strict-inclusion-f}), and the same argument used in the proof of Lemma
\ref{lem-rec-I}, we obtain $f\in \pol(\,\zeta)$ and $f\notin \pol(\rho)$. Therefore, $\pol(\rho)\cap \pol(\theta)\varsubsetneq \pol(\theta)\cap
\pol(\,\zeta).$ It is obvious  that $\pol(\theta)\cap \pol(\,\zeta)\varsubsetneq \pol(\theta)$ (because $\,\zeta\neq E_{k}^{h}$). Given
  $\left(u_1, u_2, \ldots, u_h\right)\in E_{k}^{h}\setminus\, \zeta$ and $(a,b)\notin
\theta$. $f\in \pol(\theta)$ and $f\notin \pol(\,\zeta)$; thus $\pol(\rho)\cap \pol(\theta)\varsubsetneq \pol(\,\zeta)\cap \pol(\theta)\varsubsetneq
\pol(\theta)$.
\end{proof}


\begin{lemma}\label{lem-rec-I-gamma4}
If $\,\zeta=E_{k}^{h}$ and $\,\zeta^k\neq E_k^{k}$, then $\pol(\theta)\cap \pol(\rho)$ is not maximal in $\pol(\theta)$ where $\,\zeta^{l}$ is the $l$-ary
relation on $E_k$ specified by

\begin{eqnarray*}
  \,\zeta^{l} &=& \left\{\left(a_1, \ldots,  a_l\right)\in E_{k}^{l} / \exists
u_{i}\in [a_{i}]_{\theta}/ \forall i\in \{1;\cdots;l\},\right. \\
   && \hspace{1cm}\left.\{u_1; \ldots ;u_{i-1};a_i;u_{i+1}; \ldots ;u_l\}^{h}\subseteq\rho\right\}.
\end{eqnarray*}
\end{lemma}
\begin{proof} We will prove that
$\pol(\rho)\cap \pol(\theta)\varsubsetneq \pol(\,\zeta^k)\varsubsetneq \pol(\theta)$.\\
It is easy to see that $\pol(\rho)\cap \pol(\theta)\subseteq \pol(\,\zeta^k)$.\\
Given $\left( u_1, u_2, \ldots, u_h\right)\in E_{k}^{h}\setminus
\rho$ and $(a,b)\notin \theta$, using the $h$-ary operation $f$
defined by (\ref{function-for-strict-inclusion-f}), we have again
$f\notin  \pol(\rho)$ and $f\in \pol(\,\zeta^{k})$; indeed, since
$\left( u_1, u_2, \ldots, u_h\right)\in E_{k}^{h}=\,\zeta$, there
exist $  v_i \in [u_i]_{\theta}, i\in\{1;\cdots;h\}$ such that for
all $ i\in\{1;\cdots;h\}$, $\{u_i,v_{i_1}, \ldots
,v_{i_{h-1}}\}^{h}\subseteq \rho $, with $ \{i_1,i_2, \ldots
,i_{h-1}\}=\{1;\cdots;h\}\setminus \{i\} $.
 and it  follows that $\{u_1,u_2, \ldots ,u_h\}^{k}\subseteq \,\zeta^k$.\\
Moreover, as we have $\zeta^{k}\neq E_{k}^{k}$ then $E_{k}^{k}\setminus   \,\zeta^{k}\neq \emptyset$. Given $\left(w_{1}, \ldots, w_{k}\right)\in
E_k^{k}\setminus \, \zeta^{k} $,  $(a,b)\notin \theta$, using the $k$-ary operation $h$ defined on $E_k$, by:
$$ h(x_1, \ldots ,x_k)=  \left\{\begin{array}{ll}
                                                 w_1 &\text{ if } (x_1, \ldots ,x_k)\theta (a, \ldots ,a,a)=a_1 \\
                                                 w_2 &\text{ if } (x_1, \ldots ,x_k)\theta (a,b,a, \ldots ,,a)=a_2 \\
                                                 w_3 &\text{ if } (x_1, \ldots ,x_k)\theta (a,a,b,a \ldots ,a)=a_3 \\
                                                 &\vdots\\
                                                 w_{k} &\text{ if } (x_1, \ldots ,x_k)\theta (a,a,a, \ldots a,b)=a_k\\
                                                 w_k &\text{ otherwise; }
                                               \end{array}
      \right.
$$
 and the same argument used for the similar operation $h$ in the proof of Lemma \ref{lem-rec-II} yields  $h\in \pol(\theta)$ and $h\notin
\pol(\zeta^{k})$.

Thus $\pol(\rho)\cap \pol(\theta)\varsubsetneq \pol(\,\zeta^k)\cap \pol(\theta)\varsubsetneq \pol(\theta)$.
\end{proof}

Now, we continue our induction process with the assumption  $\,
\varsigma^{l}= E_k^{l}$ for all $h\leq l\leq t$ and $\,\zeta^{k}=
E_k^{k}$. Since
 $\left(0, \ldots,  k-1 \right)\in E_{k}^{k}$, there exist $ u_{i}\in
[i]_{\theta}$ such that $$\{u_0;u_1; \ldots ;u_{i-1};i;u_{i+1};
\ldots ;u_{k-1}\}^{h}\subseteq\rho$$ for all $i\in E_{k}$. We set
$$T_i:= \min (C_i)\cap\{u_0;u_1; \ldots ;u_{k-1}\}$$ for all $i\in E_{k}$. Therefore
$\{T_{0};\cdots; T_{t-1}\}$ is a transversal of order $1$.

Before the main induction, we define the sequence $(^{h}\beta_{n})$
by:  $ ^{h}\beta_0=\, \eta $, $^{h}\beta_1=\, \rho_{0,\theta} $, $
^{h}\beta_2=\, \varsigma $, and for $l\geq 3$,

$$^{h}\beta_{l}=\underset{\sigma\in\mathcal{S}_{h}}\bigcap
(\rho_{l-1,\theta})_{\sigma}.$$

Let $n\in\{1;\cdots;h-1\}$ and assume that there exists a
transversal $T$ of order $ n-1$ for the $\theta$-classes. Set
$T=\{u_0;u_1;\cdots;u_{t-1}\}$; then for every
$a_1,a_2,...a_{n-1}\in E_{k}$, $\{a_1;a_2;\cdots
a_{n-1};u_0;u_1;\cdots; u_{t-1}\}^{h}\subseteq \rho$ and
$\rho\subseteq\,^{h}\beta_{n+1}$.
\begin{lemma}\label{lem-rec-I-gamma5-egal}
If $\rho=\, ^{h}\beta_{n+1}\varsubsetneq E_{k}^{h}$, then
$\pol(\theta)\cap \pol(\rho)$ is maximal in $\pol(\theta).$
\end{lemma}
\begin{proof}
If $\rho= \, ^{h}\beta_{n+1}\varsubsetneq E_{k}^{h}$, then $\rho$ is weakly $\theta$-closed of order $n$. Hence $\rho$ is a relation of type III. Using
Proposition \ref{sufficiency-direction} $\pol(\theta)\cap \pol(\rho)$ is maximal in $\pol(\theta)$.
\end{proof}
\begin{lemma}\label{lem-rec-I-gamma5}
If $\rho\varsubsetneq\, ^{h}\beta_{n+1}\varsubsetneq E_{k}^{h}$,
then $\pol(\theta)\cap \pol(\rho)$ is not maximal in $\pol(\theta).$
\end{lemma}
\begin{proof}We have to prove that
$\pol(\rho)\cap \pol(\theta)\varsubsetneq \pol(\theta)\cap \pol(\, ^{h}\beta_{n+1})\varsubsetneq \pol(\theta)$.\\
It is easy to show that $\pol(\rho)\cap \pol(\theta)\subseteq \pol(\theta)\cap \pol(\, ^{h}\beta_{n+1})$.\\
Using the $h$-ary operation $f$ defined by (\ref{function-for-strict-inclusion-f}), we show that $\pol(\rho)\cap \pol(\theta)\varsubsetneq \pol(\,
^{h}\beta_{n+1})\cap \pol(\theta)\varsubsetneq \pol(\theta)$.
\end{proof}
\begin{lemma}\label{lem-rec-II-gamma5}\text{ }
If $\rho\varsubsetneq\, ^{h}\beta_{n+1}= E_{k}^{h}$  and  $\, ^{h}\beta_{n+1}^{l}\neq E_k^{l}$ for some $l>h$, then $\pol(\theta)\cap \pol(\rho)$
is not maximal in $\pol(\theta)$;\\
where $ ^{h}\beta_{n+1}^{l}=\underset{\sigma\in\mathcal{S}_{l}}\bigcap(\rho_{n,\theta}^{l})_{\sigma}$ with
\begin{eqnarray*}
  \, \rho_{n,\theta}^{l} &:=& \left\{\left(a_1,
\ldots,  a_l \right)\in E_{k}^{l} / \exists u_{i}\in [a_{i}]_{\theta}, n+1\leq i\leq l: \right. \\
   && \,\,\,\,\,\,\,\,\,\,\,\,\,\,\,\,\left.\{a_1;a_2;\cdots;a_n;u_{n+1},\ldots,u_{l}\}^{h}\subseteq \rho\right\}.
\end{eqnarray*}
\end{lemma}
\begin{proof}
Denote $m:=\min\{l\in \mathbb{N}\setminus\{0;1;2; \ldots
;h\}/^{h}\beta_{n+1}^{l}\neq E_k^{l}\}$. Before proving that
$\pol(\rho)\cap \pol(\theta)\varsubsetneq \pol(\,
^{h}\beta_{n+1}^{m})\cap \pol(\theta)\varsubsetneq \pol(\theta)$, we
will first
prove that $\pol(\rho)\cap \pol(\theta)\subseteq \pol(\, ^{h}\beta_{n+1}^{m})\cap \pol(\theta)$.\\
Let $f\in \pol(\theta)\cap \pol(\rho)$  be an $p$-ary operation on $E_k$. Let
$$\left(a_{11}, \ldots,  a_{m1}\right), \ldots ,\left(a_{1p}, \ldots, a_{mp} \right)\in \, ^{h}\beta_{n+1}^{m}.$$
By definition of $\, ^{h}\beta_{n+1}^{m}$, for all $ c\in \{1;\cdots;p\}$, for all $\sigma\in\mathcal{S}_{m}$, there exist $ u_{\sigma(r)c}\in
[a_{\sigma(r)c}]_{\theta}, n+1 \leq r\leq m$  such that  $$ \{ a_{\sigma(1)c};\ldots; a_{\sigma(n)c}; u_{\sigma(n+1)c} ;\ldots;u_{\sigma(m)c}
\}^{h}\subseteq \rho,$$ and from the fact that $f\in \pol(\rho)$ it follows that for all $ \sigma\in\mathcal{S}_{m}$,
\[\begin{array}{ccr}
    \{f(a_{\sigma(1)1}, \ldots ,a_{\sigma(1)p}) ;\ldots; f(a_{\sigma(n)1}, \ldots ,a_{\sigma(n)p}); & &  \\
    \hspace{2cm}f(u_{\sigma(n+1)1}, \ldots ,u_{\sigma(n+1)p});\ldots; f(u_{\sigma(m)1}, \ldots,u_{\sigma(m)p})\}^{h}  &\subseteq  & \rho.
  \end{array}
\]
Since $f\in \pol(\theta)$, it follows that for all $ i\in \{n+1;\cdots;m\}$,  and for all $ d\in \{n+1;\cdots;m\}$,
$$\sigma(d)=i \Rightarrow f(u_{\sigma(d)1},u_{\sigma(d)2}, \ldots ,u_{\sigma(d)p})\theta
f(a_{i1}, \ldots ,a_{ip}).$$
  From the definition of $\, ^{h}\beta_{n+1}^{m}$, it follows that $$\left(f(a_{11}, \ldots ,a_{1p}), \ldots, f(a_{m1}, \ldots ,a_{mp})\right)\in \, ^{h}\beta_{n+1}^{m}.$$
Then $f\in \pol(\, ^{h}\beta_{n+1}^{m})\cap \pol(\theta)$.\\
Fix $\left(v_{1}, \ldots, v_{h}\right)\in E_{k}^{h}\setminus \rho$ and  let $(a,b)\notin \theta$. Using the $h$-ary operation $f$ defined by
(\ref{function-for-strict-inclusion-f})(where we replace $u_{i}$ by $v_{i}$, $1\leq i \leq h$), we have $f\notin \pol(\rho)$ and $f\in \pol(\,
^{h}\beta_{n+1}^{m})$. Indeed, due to the fact that
 $$\left(v_{1}, \ldots, v_{h}\right)\in E_{k}^{h}=\, ^{h}\beta_{n+1},$$
 for all $\sigma\in\mathcal{S}_{h}$ there exist $d_{r}\in [v_{r}]_{\theta}$, $n+1\leq r\leq h$ such that,
 $$ \{ v_{\sigma(1)};\ldots; a_{\sigma(n)};  d_{\sigma(n+1)};\ldots;d_{\sigma(h)} \}^{h}\subseteq \rho.$$
Therefore,  it follows that $\{v_1;v_2; \ldots ;v_h\}^{m}\subseteq \, ^{h}\beta_{n+1}^m$. \\
Moreover, as we have $\, ^{h}\beta_{n+1}^{m}\neq E_{k}^{m}$, then $ E_k^{m}\setminus \, ^{h}\beta_{n+1}^{m}$ is not an empty relation. Given $\left(v_{1},
\ldots, v_{m}\right)\in E_k^{m}\setminus \, ^{h}\beta_{n+1}^{m}$ and  $(a,b)\notin \theta$; using the $m$-ary operation $h$ defined by
(\ref{function-for-strict-inclusion-h-in-theta}) (where we replace $u_{i}$ by $v_{i}$, $1\leq i \leq m$), and the same argument used in the proof of Lemma
\ref{lem-rec-II}, it follows that $h\in \pol(\theta)$ and $h\notin \pol(\, ^{h}\beta_{n+1}^{m})$. Thus $\pol(\rho)\cap \pol(\theta)\varsubsetneq \pol(\,
^{h}\beta_{n+1}^{m})\cap \pol(\theta)\varsubsetneq \pol(\theta)$.
\end{proof}

It is naturally now to suppose that for each $l\in\{h;\cdots; t\}$,
$^{h}\beta_{n+1}^{l}= E_k^l$.

Let $m_1,...,m_n$ be integers such that $m_{1}>m_{2}>\dots>m_{n}$
and for each $i\in\{0;\cdots; t-1\}$, $|C_i|\leq
\min\{m_1,...,m_n\}$ or $|C_i|\in \{m_1,...,m_n\}$. We set
$m=m_1+...+m_n$. In this part we will use these notations
$$\mathbf{a}:=\left(a_1, \ldots, a_{m_1},a_{m_{1}+1}, \ldots,
a_{m_1+m_2},\ldots, a_{m_1+...+m_n},a_{m+1}, \ldots,  a_{m+t-n}
\right),$$
$$\Pi=\{1;\cdots;m_1\}\times\{m_1+1;\cdots;m_1+m_2\}\times...\times\{m_1+\cdots+m_{n-1}+1;\cdots;m\},$$ and
$$\Pi'=\{1;\cdots;m_1\}^{2}\cup\{m_1+1;\cdots;m_1+m_2\}^{2}\cup...\cup\{m_1+\cdots+m_{n-1}+1;\cdots;m\}^{2}.$$

Let $\, ^{h}\beta_{n+1}^{'}$ be the relation
\begin{eqnarray*}
  \, ^{h}\beta_{n+1}^{'}&=&\left\{\mathbf{a}\in E_{k}^{m+t-n} / \forall (i,j)\in \Pi', a_i\theta a_j;\right. \\
   &&\exists u_{i}\in [a_{i}]_{\theta},  m+1\leq i \leq m+t-n,\text{ such that } \\
&& \left.   \forall(i_1,...,i_n)\in\Pi, \{ a_{i_1}, \ldots, a_{i_{n}},u_{m+1}, \ldots ,u_{m+t-n}\}^{h}\subseteq \rho.\right\}
\end{eqnarray*}
 We have $\pol(\, ^{h}\beta_{n+1}^{'})\subseteq \pol(\theta)$.

We define two relations $\varrho_{1}$ and $\varrho_{2}$ on $\{1;\cdots;m+t-n\}$ by:
$$(i,j)\in\varrho_{1}\Leftrightarrow i=j \text { and } (i,j)\in \varrho_{2}\Leftrightarrow (i=j \text{ or } (i,j)\in \Pi').$$
$\varrho_{1}$ and $\varrho_{2}$ are  equivalence relations and $\, ^{h}\beta_{n+1}^{'} \subseteq D_{\varrho_{1}\varrho_{2}}$.

\begin{lemma}\label{lem-rec-II-beta-n'}\text{ }
If $\forall l\in\{h,h+1,...,t\}$, $\, ^{h}\beta_{n+1}^{l}=
E_{k}^{h}$ and $\, ^{h}\beta_{n+1}^{'}\neq
D_{\varrho_{1}\varrho_{2}}$, then $\pol(\theta)\cap \pol(\rho)$ is
not maximal in $\pol(\theta)$.
\end{lemma}
\begin{proof}
It is similar to the proof of Lemma \ref{lem-rec-II-gamma3'}.
\end{proof}

In what follows, we assume that
\begin{equation}\label{transversale-classe-ind}
\, ^{h}\beta_{n+1}^{'}= D_{\varrho_{1}\varrho_{2}}.\end{equation} i.e.,  for all $i_1,...,i_n\in\{0;\cdots; t-1\},$ there exist $u_{jI}\in C_j,
j\in\{0;\cdots; t-1\}\setminus \{i_1,...,i_n\}$, with $I=\{i_1,...,i_n\}$ such that for all $a_{i_{r}}\in C_{i_{r}}, r\in\{1;\cdots;n\}$ we have
$$\{a_{i_1};\cdots;a_{i_{n}},u_{1I}; \ldots ;u_{t-n \,I}\}^{h}\subseteq\rho.$$\\
We obtain a transversal of order $n$ for the $\theta$-classes. Hence
we can continue the previous induction until $$\,
^{h}\beta_{h}=\underset{\sigma\in\mathcal{S}_{h}}\bigcap(\rho_{h-1,\theta})_{\sigma}=
E_{k}^{h}.$$ From now, we suppose that $\, ^{h}\beta_{h}=
E_{k}^{h}$.

\begin{lemma}\label{lemma-reciproque-gamma5}
If $\, ^{h}\beta_{h}= E_{k}^{h}$ and there exists $l>h $ such that $\, ^{h}\beta_{h}^{l}\neq E_{k}^{h}$, then $\pol(\rho)\cap \pol(\theta)$ is not maximal
in $\pol(\theta)$.
\end{lemma}
\begin{proof}
Set $m=\min\{l\in\{1;\cdots;k\}\setminus\{1;\cdots;h\}/ \, ^{h}\beta_{h}^{l}\neq E_{k}^{h}\}$. It is easy to prove that $\pol(\rho)\cap
\pol(\theta)\varsubsetneq \pol(\, ^{h}\beta_{h}^{m})\cap \pol(\theta)\varsubsetneq \pol(\theta)$.
\end{proof}

\begin{lemma}\label{lemma-classe-central}\text{ }
If $\, ^{h}\beta_{h}^{k}= E_k^{k}$, then $\pol(\theta)\cap \pol(\rho)$ is
 maximal in $\pol(\theta)$.
\end{lemma}
\begin{proof}
If $\, ^{h}\beta_{h}^{k}= E_k^{k}$, then each equivalence class of
$\theta$ has a central element of $\rho$. Since $\,
\eta\subseteq\rho$, $\rho$ is of type I and
$\pol(\rho)\cap\pol(\theta)$ is maximal in $\pol(\theta)$ by
Proposition \ref{sufficiency-direction}.
\end{proof}

We are ready now to give the proof of Proposition
\ref{completeness-criterion} and Theorem \ref{maintheorem}.
\begin{proof}[Proof of Proposition \ref{completeness-criterion}]
Combining Lemmas \ref{major-criterion}, \ref{lem-rec-I}, \ref{lem-rec-II}, \ref{lem-rec-I-gamma3}, \ref{lem-rec-II-gamma3}, \ref{lem-rec-II-gamma3'},
\ref{lem-rec-O-gamma4}, \ref{lem-rec-I-gamma4}, \ref{lem-rec-I-gamma5}, \ref{lem-rec-II-gamma5}, \ref{lemma-reciproque-gamma5}, \ref{lemma-classe-central}
and Remark \ref{arity-criterion}, we obtain the result.
\end{proof}

\begin{proof}[Proof of Theorem \ref{maintheorem}]
It follows from Propositions \ref{sufficiency-direction} and
\ref{completeness-criterion}.
\end{proof}

Fourthly, we look at $h$-regular relations.
\subsection{$h$-regular relations}\label{sec7}

As an $h$-regular relation is totally reflexive and totally
symmetric, some results state in the previous subsection can be
applied to the $h$-regular relation. Besides an $h$-regular
generated relation does not contain a central element. We will prove
that there is submaximality if and only if $\rho$ is
$\theta$-closed(or of type II). We begin this subsection with some
examples of $h$-regular relations.
 \begin{example}
Let $k\geq 3$ be an integer and $0\leq i<j<r<n\leq k-1$, we denote
by $A_{i,j,r}$ and $A_{i,j,r,n}$ the sets

$$A_{i,j,r}:=\{(\sigma(i),\sigma(j),\sigma(r)); \sigma\in\mathcal{S}_{\{i;j;r\}}\} $$ and
$$A_{i,j,r,n}:=\{(\sigma(i),\sigma(j),\sigma(r),\sigma(n)); \sigma\in\mathcal{S}_{\{i;j;r;n\}}\}.$$

We consider the following equivalence relations $\theta_{6}$ $\theta_{7}$ $\theta_{8}$ on $E_{12}$ defined respectively by their equivalence classes
denoted by $C_{m}^{i}, 6\leq i\leq 8$ as follows:

$$ C_{0}^{6}=\{0;1;2;3;4\}, C_{1}^{6}=\{5;6;7\}, C_{2}^{6}=\{8;9;10;11\};$$

$$ C_{0}^{7}=\{0;1;5;8\}, C_{1}^{7}=\{2;6;9;11\}, C_{2}^{7}=\{3;4;7;10\};$$

$$ C_{0}^{8}=\{0;1\}, C_{1}^{8}=\{2\}, C_{2}^{8}=\{3;4\}, C_{3}^{8}=\{5\}, C_{4}^{8}=\{6\}, C_{5}^{8}=\{7\},$$ $$C_{6}^{8}=\{8\}, C_{7}^{8}=\{9\},
C_{8}^{8}=\{10\}, C_{9}^{8}=\{11\};$$

and the relation

$$\Upsilon_{6}=\{(x_1,x_2,x_3)\in E_{12}^{3}: (x_1,x_2)\in\theta_{6},\textit{ or } (x_2,x_3)\in\theta_{6},\textit{ or } (x_1,x_3)\in\theta_{6}\}.$$
$$\Upsilon_{7}=\underset{\sigma\in\mathcal{S}_{3}}\cup(\Upsilon)_{\sigma}$$ where
$$\Upsilon=\{(x_1,x_2,x_3)\in E_{12}^{3}: (x_1,x_2)\in\theta_{6}, (x_2,x_3)\in\theta_{7}\}.$$

It is easy to see that $\Upsilon_{6}$ is a $\theta_{6}$-closed
$3$-regular relation associated to $T=\{\theta_{6}\}$, and
$\Upsilon_{7}$ is a $\theta_{8}$-closed $3$-regular relation
associated to $T=\{\theta_{6};\theta_{7}\}$.
\end{example}
We continue with the characterization of $\theta$-closed $h$-regular
relation.

\begin{lemma}\label{caracterisation-relation-h-reguliere-theta-fermee}
For $h\geq3$, let $\rho$ be an $h$-regular relation on $E_{k}$
determined by the $h$-regular family $T =
\{\theta_{1};\cdots;\theta_{m}\}$ and $\theta$ a nontrivial
equivalence relation on $E_{k}$. $\rho$ is  $\theta$-closed iff
$\theta\subseteq\theta_{i}$, for all $1\leq i\leq m$.
\end{lemma}
\begin{proof}
$\Rightarrow)$ Firstly, we show that $\eta=\{(a_{1},\dots,a_{h})\in
E_{k}^{h}:\ \ (a_{1},a_{2})\in\theta\}\subseteq\rho$. Let
$(a_{1},\dots,a_{h})\in\eta$; since $\rho$ is totally reflexive, we
have $(a_{1},a_{1},a_{3},\dots,a_{h})\in\rho$ and
$(a_{1},\dots,a_{h})\theta(a_{1},a_{1},a_{3},\dots,a_{h})$. Hence
$(a_{1},\dots,a_{h})\in\rho_{0,\theta}=\rho$. Our next step is to
show that $\theta\subseteq\theta_{i}$ for all $1\leq i\leq m$. Let
$i\in\{1,\dots,m\}$ and $(a,b)\in\theta$; set
$A_{i}=E_{k}/\theta_{i}\setminus\{[a]_{\theta_{i}},[b]_{\theta_{i}}\}$.
It is easy to see that $|A_{i}|\geq h-2$; choose
$(a_{1},\dots,a_{h-2})\in E_{k}^{h-2}$ such that
$[a_{p}]_{\theta_{i}}\in A_{i}$ for all $1\leq p\leq h-2$ and
$(a_{p},a_{q})\notin\theta_{i}$ for all $1\leq p<q\leq h-2$. Due to
$\eta\subseteq\rho$, we have $(a,b,a_{1},\dots,a_{h-2})\in\rho$;
therefore $(a,b)\in\theta$ and $\theta\subseteq\theta_{i}$.

$\Leftarrow)$ It follows from the fact that
$\theta\circ\theta_{i}\circ\theta\subseteq\theta_{i}$ for all $1\leq
i\leq m$.
\end{proof}

\begin{lemma}\label{caracterisation-relation-h-reguliere-theta-fermee-2}
For $h\geq3$, let $\rho$ be an $h$-ary relation and $\theta$ a
nontrivial equivalence relation on $E_{k}$. $\rho$ is a
$\theta$-closed $h$-regular relation
    iff there exists an $h$-regular relation $\psi$ on $E_{t}$ such
that $\rho=\varphi^{-1}(\psi)$.
\end{lemma}
\begin{proof}
Firstly, let us assume that there exists a $h$-regular relation $\psi$ on $E_{t}$ and put $\bot=\{\nu_{1};\cdots;\nu_{n}\}$ the $h$-regular family
associated to $\psi$. Clearly $\varphi^{-1}(\bot)=\{\varphi^{-1}(\nu_{1});\cdots;\varphi^{-1}(\nu_{n})\}$ is a $h$-regular family and
$\rho=\varphi^{-1}(\psi)$ is exactly the $h$-regular relation associated to $\varphi^{-1}(\bot)$. Moreover, for all $1\leq i\leq n$ we have
$\theta\subseteq\varphi^{-1}(\nu_{i})$. Therefore, by Lemma \ref{caracterisation-relation-h-reguliere-theta-fermee}, $\rho$ is $\theta$-closed.\\
Conversely, assume that $\rho$ is a $h$-regular relation
 determined by the $h$-regular family $T = \{\theta_{1};\cdots\theta_{r}\}$ and $\rho$ is $\theta$-closed. Clearly, for all
$1\leq i\leq r$, $\varphi(\theta_{i})$ is an equivalence relation with exactly $h$ equivalence classes whose are the images of equivalence classes of
$\theta_{i}$ by $\varphi$. It follows that $\cap\{\varphi(B_i)| 1\leq i\leq r\}$ is non-empty for arbitrary equivalence classes $\varphi(B_i)$ of
$\varphi(\theta_{i})$ , $1\leq i\leq r$. Therefore $\varphi(T)= \{\varphi(\theta_{1});\cdots\varphi(\theta_{r})\}$ is a $h$-regular family and
$\varphi(\rho)$ is a $h$-regular relation on $E_{t}$, associated to $\varphi(T)$. Moreover $\varphi^{-1}(\varphi(\rho))=\rho$ because $\rho$ is
$\theta$-closed. Therefore we have the result.
\end{proof}

We end this subsection with the characterization of $h$-regular
relations $\rho$ such that $\pol(\theta)\cap\pol(\rho)$ is maximal
in $\pol(\theta)$. We need the next lemma to prove our main result.

\begin{lemma}\cite{LAU}\label{caracterisation-relation-h-reguliere}
Let $\chi\subseteq E_{t}^{h}$ such that $\pol(\chi)$ is maximal in $\mathcal{L}_{t}$. Then $\pol(\theta)\cap\pol(\varphi^{-1}(\chi))$ is submaximal in
$\pol(\theta)$.
\end{lemma}
\begin{proof}
See the proof of Lemma 18.2.5 Page 565 in \cite{LAU}.
\end{proof}
\begin{lemma}\label{major-criterion-regular}
If $\pol(\rho)\cap \pol(\theta)$ is maximal in $\pol(\theta)$, then
$\eta\subseteq \rho$.
\end{lemma}
\begin{proof}
See the proof of Lemma \ref{major-criterion}.
\end{proof}
Now we give the main result of this subsection.

\begin{proposition}\label{reg-max-in-equiv-completeness}
Let $\theta$ be a nontrivial equivalence relation and  $\rho$ be an
$h$-regular relation on $E_{k}$ determined by the $h$-regular family
$T = \{\theta_{1};\cdots;\theta_{m}\}$. $\pol(\theta)\cap
\pol(\rho)$ is maximal in $\pol(\theta)$ if and only if $\rho$ is
$\theta$-closed.
\end{proposition}

\begin{proof}
$\Rightarrow)$ It follows from Lemma \ref{major-criterion-regular}
that $\eta\subseteq\rho$; and in the light of the proof of Lemma
\ref{caracterisation-relation-h-reguliere-theta-fermee}, we conclude
that $\theta\subseteq\theta_{i}$ for all $1\leq i\leq m$. It follows
from Lemma \ref{caracterisation-relation-h-reguliere-theta-fermee}
that $\rho$ is $\theta$-closed.

$\Leftarrow)$ Combining Lemmas
\ref{caracterisation-relation-h-reguliere-theta-fermee},
\ref{caracterisation-relation-h-reguliere-theta-fermee-2},
\ref{caracterisation-relation-h-reguliere}, we have the result.
\end{proof}

\section{Conclusion}\label{sec8}
In this paper, we characterized the relations $\rho$ from the
Rosenberg's list for which $\pol(\theta) \cap \pol(\rho)$ is maximal
in $\pol(\theta)$, where $\theta$ is a nontrivial equivalence
relation. The classification of all central relations $\rho$ on
$E_k$ such that the clone $\pol(\theta) \cap \pol(\rho)$ is maximal
in $\pol(\theta)$ improves Temgoua and Rosenberg's results
\cite{TEM-ROS}. We plan in a future project to characterize the
meet-irreducible submaximal clones of $\pol(\theta)$ for a
nontrivial equivalence relation $\theta$.



\begin{thebibliography}{99}
\bibitem{BA-PI} Baker, K. A.,  Pixley, A. F.:  Polynomial interPolation and the Chinese remainder theorem for algebraic systems. Math. Z. \textbf{143}, 165-174 (1975)

\bibitem{BUR} Burris, S.,  Sankappanavar, H.P.: A Course in Universal Algebra. Springer-Verlag (1981)

\bibitem{DEN} Denecke, K.,  Wismath, W.L.:  Universal Algebra and Applications in Theoretical Computer Science. Chapman \& Hall / CRC (2002)

\bibitem{GRE} Grecianu, A.:  Clones sous-maximaux inf-r\'{e}ductibles. M\'{e}moire num\'{e}ris\'{e} par la Division de la gestion de
documents et des archives de l'Universit\'{e} de Montr\'{e}al http://hdl.handle.net/1866/7887 (2009)

\bibitem{JAB} Jablonskij, S.V.:   Functional constructions in many-valued logics. Tr. Mat. Inst. Steklova \textbf{51}, 5-142 (1958) (Russian)

\bibitem{LAR} Larose, B.:  Maximal subclones of the clone of isotone operations on a bounded finite order. C. R. Acad. Sci. Paris,  \textbf{314}, 245-247 (1992)

\bibitem{LAU} Lau, D.:  Function algebras on finite sets. Springer (2006)

\bibitem{PIN} Pinsker, M.:  Rosenberg's characterization of maximal clones. Diploma thesis, Vienna Univ. of technology (2002)

\bibitem{POES} Poeschel, R.:  Concrete representation of algebriac structure and a general Galois theory. Contribution to General Algebra, (Proc. Klagenfurt Conf., Klagenfurt, 1978), Verlag J. Heyn, Klagenfurt, \textbf{1}, 249 - 272
(1979)

\bibitem{PONJ} Ponjanie, M.:  On traces of Maximal clones. Master thesis, Univ. of  Novi Sad (2003)

\bibitem{PO21} Post, E.L.:  Introduction to a general theory of elementary propositions. Amer. J. Math. \textbf{43}, 163-185 (1921)

\bibitem{PO41} Post, E. L.:  The two-valued iterative systems of mathematical logic. Ann. Math. Studies, vol. \textbf{5}. Princeton University Press (1941)%

\bibitem{QUACK} Quackenbush, R.W.:   A new proof of Rosenberg's primal algebra characterization theorem. In: Finite algebra and multiple valued logic (Szeged, 1979). Colloq. Math. Soc. J. Bolyai, North Holland, Amsterdam \textbf{28}, 603 - 634 (1981).

\bibitem{ROS65} Rosenberg, I.G.:   La structure des fonctions de plusieurs variables sur un ensemble fini. C. R. Acad. Sci. Paris, Ser. A-B, \textbf{260}, 3817-3819 (1965)

\bibitem{ROS70} Rosenberg, I.G.:   Functional completeness in many-valued logics; the structure of functions of several variables on finite sets. Rozpravy.
\v{C}eskoslovenské Akad. V\v{e}d \v{R}ada Mat. P\v{r}\'{i}rod. V\v{e}d \textbf{80}, 3-93 (1970) (German)

\bibitem{ROS70-2} Rosenberg, I.G.  Über die functionals Vollst$\ddot a$ndigkeit in dem mehrwertigen Logiken.
Rozpravy \~Ceskoslovenk\'e Akad. v\~ed. Ser. Math. Nat. Sci. \textbf{80}, 1970; 3 -93.

\bibitem{ROS85} Rosenberg, I.G.,  Szendrei, A.:  Submaximal clones with  a prime order automorphism.  Acta Sci. Math., \textbf{49}, 29 -48 (1985)

\bibitem{TEM-ROS}  Temgoua, E. R. A., Rosenberg, I. G.:  Binary central relations and submaximal clones determined by nontrivial equivalence relations.
Algebra Universalis, Springer Basel AG, \textbf{67}, 299-311 (2012)

\bibitem{TEM-meet-irred} Temgoua, E.R.A.: Meet-irreducible submaximal clones determined by nontrivial equivalence relations. Algebra Universalis \textbf{70}, 175-196 (2013).


\end{thebibliography}
\end{document}